\documentclass[a4paper,12pt]{article}
\usepackage{amsmath,amsthm,amssymb,latexsym,epsfig}
\title{{\bf Estimates for the maximal singular integral in terms of the singular integral: the case of even kernels}}
\author{\Large{\Large Joan Mateu, Joan Orobitg and Joan Verdera}}
\newtheorem*{teor}{Theorem}
\newtheorem{co}{Corollary}
\newtheorem{lemma}[co]{Lemma}
\newtheorem*{CZ}{Lemma}
\newtheorem*{lemanom}{Lemma}

\theoremstyle{definition}
\newtheorem{example}{Example}
\newtheorem{question}{Question}
\newtheorem{remark}{Remark}
\newtheorem*{gracies}{Acknowledgements}

\newcommand{\Rn}{{\mathbb R}^n}

\newcommand{\C}{\mathbb{C}}
\newcommand{\BC}{\Rn \setminus \overline B}
\newcommand{\n}{\frac{n}{2}}

\begin{document}

\date{}

\maketitle

\begin{abstract}
Let $T$ be a smooth homogeneous Calder\'{o}n-Zygmund singular integral
operator in~$\Rn$. In this paper we study the problem of controlling
the maximal singular integral~$T^{\star}f$ by the singular
integral~$Tf$. The most basic form of control one may consider is
the estimate of the $L^2(\Rn)$ norm of $T^{\star}f$ by a constant
times the $L^2(\Rn)$ norm of $Tf$.  We show that if $T$ is an even
higher order Riesz transform, then one has the stronger pointwise
inequality $T^{\star}f(x) \leq  C \, M(Tf)(x)$, where $C$ is a
constant and $M$ is the Hardy-Littlewood maximal operator. We prove
that the $L^2$ estimate of $T^{\star}$ by $T$ is equivalent, for
even smooth homogeneous Calder\'{o}n-Zygmund operators, to the pointwise
inequality between $T^{\star}$ and $M(T)$. Our main result
characterizes the $L^2$ and pointwise inequalities in terms of an
algebraic condition expressed in terms of the kernel
$\frac{\Omega(x)}{|x|^n}$ of $T$, where $\Omega$ is an even
homogeneous function of degree~$0$, of class $C^\infty(S^{n-1})$ and
with zero integral on the unit sphere $S^{n-1}$. Let $\Omega= \sum
P_j$ the expansion of $\Omega$ in spherical harmonics $P_j$ of
degree~$j$. Let $A$ stand for the algebra generated by the identity
and the smooth homogeneous Calder\'{o}n-Zygmund operators. Then our
characterizing condition states that $T$ is of the form~$R\circ U$,
where $U$ is an invertible operator in $A$ and $R$ is a higher order
Riesz transform associated with a homogeneous harmonic
polynomial~$P$ which divides each $P_j$ in the ring of polynomials
in $n$~variables with real coefficients.
\end{abstract}

\section{Introduction}

Let $T$ be a smooth homogeneous Calder\'{o}n-Zygmund singular integral
operator on~$\Rn$ with kernel
\begin{equation}\label{eq1}
K(x)=\frac{\Omega(x)}{|x|^n},\quad x \in \Rn \setminus \{0\}\,,
\end{equation}
where $\Omega$ is a (real valued) homogeneous function of degree $0$
whose restriction to the unit sphere $S^{n-1}$ is of
class~$C^\infty(S^{n-1})$ and satisfies the cancellation property
$$
\int_{|x|=1} \Omega(x)\,d\sigma(x)=0\,,
$$
$\sigma$ being the normalized surface measure on $S^{n-1}$. Recall
that $Tf$ is the principal value convolution operator
\begin{equation}\label{eq2}
Tf(x)= P.V. \int f(x-y)\,K(y)\,dy \equiv \lim_{\epsilon\rightarrow
0} T^\epsilon f(x)\,,
\end{equation}
where $T^\epsilon$ is the truncation at level $\epsilon$ defined by
$$
T^{\epsilon}f(x)= \int_{| y-x| > \epsilon} f(x-y) K(y) \,dy\,.
$$
As we know, the limit in \eqref{eq2} exists for almost all $x$ for
$f$ in $L^p(\Rn)$, $1 \leq p < \infty$.

Let $T^{\star}$ be the maximal singular integral
$$
T^{\star}f(x)= \sup_{\epsilon > 0} | T^{\epsilon}f(x)|, \quad x \in
\Rn\,.
$$

In this paper we consider the problem of controlling $T^{\star} f$
by $Tf$. The most basic form of control one may think of is the
$L^2$ estimate
\begin{equation} \label{eq3}
 \| T^{\star} f  \|_2 \leq  C \| T f \|_2,\quad f \in L^2(\Rn) \,.
\end{equation}
Another way of saying that $T^{\star} f$ is dominated by  $Tf$,
apparently much stronger, is provided by the pointwise inequality
\begin{equation}\label{eq4}
T^{\star}f(x) \leq C \, M(Tf)(x),\quad x \in \Rn\,,
\end{equation}
where $M$ denotes the Hardy-Littlewood maximal operator. Notice that
\eqref{eq4} may be viewed as an improved version of classical
Cotlar's inequality
\begin{equation*}\label{}
T^{\star}f(x) \leq C  \left(M(Tf)(x) + Mf (x)\right),\quad x \in
\Rn\,,
\end{equation*}
because the term involving $Mf$ is missing in the right hand side of
\eqref{eq4}.

We prove that if $T$ is an even higher order Riesz transform, then
\eqref{eq4} holds.  Recall that $T$ is a higher order Riesz
transform if its kernel is given by a function $\Omega$ of the form
$$
\Omega(x)=\frac{P(x)}{|x|^d}, \quad x \in \Rn \setminus \{0\}\,,
$$
with $P$ a homogeneous harmonic polynomial of degree $d \geq 1$. If
$P(x)= x_j$, then one obtains the $j$-th Riesz transform $R_j$. If
the homogeneous polynomial $P$ is not required to be harmonic, but
has still zero integral on the unit sphere, then we call~$T$~a
polynomial operator.

Thus, if $T=R$ is an even higher order Riesz transform, one has the
weak $L^1$~type inequality
\begin{equation}\label{eq5}
\|R^\star f \|_{1,\infty} \leq C\| Rf \|_1 \,,
\end{equation}
which combined with the classical weak $L^1$ type estimate
$$
\|R^\star f \|_{1,\infty} \leq C\| f \|_1 \,,
$$
yields the sharp inequality
$$
\| R^{\star} f  \|_{1,\infty} \leq  C  \min \{\| f \|_1, \| Rf
\|_1\}\,.
$$
In \cite{MV2} one proved \eqref{eq5} for the Beurling transform in
the plane and also that \eqref{eq5}~fails for the Riesz transforms
$R_j$. Therefore the assumption that the operator is even is
crucial.

The question of estimating $T^{\star}f$ by $Tf$ was first raised
in~\cite{MV2}. The problem originated in an attempt to gain a better
understanding of how one can obtain a. e. existence of principal
values of truncated singular integrals from $L^2$ boundedness, for
underlying measures more general than the Lebesgue measure in $\Rn$.
This in turn is motivated by a problem of David and Semmes
\cite{DS}, which consists in deriving uniform rectifiability of a
$d$-dimensional Ahlfors regular subset of $\Rn$ from the $L^2$
boundedness of the Riesz kernel of homogeneity $-d$ with respect to
$d$-dimensional Hausdorff measure on the set . For more details on
that see the last section of \cite{MaV}.

Our main result states that for even operators inequalities
\eqref{eq3} and \eqref{eq4} are equivalent to an algebraic condition
involving the expansion of $\Omega$ in spherical harmonics. This
condition may be very easily checked in practice and so, in
particular,  we can produce extremely simple examples of even
polynomial operators for which \eqref{eq3} and~\eqref{eq4} fail. For
these operators no control of $T^{\star}f$ by $Tf$ seems to be
known. To state our main result we need to introduce a piece of
notation.

Recall that $\Omega$ has an expansion in spherical harmonics, that
is,
\begin{equation}\label{eq6}
\Omega(x) = \sum_{j=1}^\infty P_j(x), \quad x \in S^{n-1}\,,
\end{equation}
where $P_j$ is a homogeneous harmonic polynomial of degree $j$. If
$\Omega$ is even, then only the $P_j$ of even degree $j$ may be
non-zero.

%The Hardy-Littlewood maximal function is
%$$
%Mf(x)= \sup_{r>0} \frac{1}{|B(x,r)|} \int_{B(x,r)} |f(y)|\,dy,
%$$
%where $B(x,r)$ is the open ball of center $x$ and radius $r$ and $f$
%is a locally integrable function.

An important role in this paper will be played by the algebra $A$
consisting of the bounded operators on $L^2(\Rn)$ of the form
$$
\lambda I + S\,,
$$
where $\lambda$ is a real number and $S$ a smooth homogeneous
Calder\'{o}n-Zygmund operator.

Our main result reads as follows.

\begin{teor} \label{T}
Let $T$ be an even smooth homogeneous Calder\'{o}n-Zygmund operator with
kernel~\eqref{eq1} and assume that $\Omega$ has the
expansion~\eqref{eq6}. Then the following are equivalent.
\begin{enumerate}
\item[(i)]
$$
T^{\star}f(x) \leq C \, M(Tf)(x),\quad x \in \Rn\,.
$$
\item [(ii)]
$$
 \| T^{\star} f  \|_2 \leq  C \| T f \|_2,\quad f \in L^2(\Rn) \,.
$$
\item [(iii)]
The operator $T$ can be factorized as $T=R\circ U$, where $U$ is an
invertible operator in the algebra $A$ and $R$ is a higher order
Riesz transform associated with a harmonic homogeneous polynomial
$P$ which divides each $P_j$ in the ring of polynomials in $n$
variables with real coefficients.
\end{enumerate}
\end{teor}

Two remarks are in order.

%$ {\bf{Remark 2.}} We do not know whether the $L^p$ inequality
%$$
 %\| T^{\star} f  \|_p \leq  C \| T f \|_p\,,\quad f \in L^p(\Rn) \,,
%$$
%is equivalent to $(ii)$ for $p \neq 2\,, \, 1<p<\infty $.

\begin{remark}
Observe that condition~$(iii)$ is algebraic in nature. This is one
of the reasons that makes the proof difficult. Condition $(iii)$ can
be reformulated in a more concrete fashion as follows. Assume that
the expansion of $\Omega$ in spherical harmonics is
$$
\Omega(x) = \sum_{j=j_0}^\infty P_{2j}(x),\quad P_{2j_0} \neq 0\,.
$$
Then $(iii)$ is equivalent to the following

\begin{enumerate}
\item[$(iv)$] {\it For each $j$ there exists a homogeneous polynomial
$Q_{2j-2j_0}$ of degree~$2j-2j_0$ such that $P_{2j}=
P_{2j_0}\,Q_{2j-2j_0}$ and $\sum_{j=j_0}^\infty \gamma_{2j}\,
Q_{2j-2j_0}(\xi) \neq 0, \quad \xi \in S^{n-1}$.}
\end{enumerate}
\end{remark}

Here for a positive integer $j$ we have set
\begin{equation}\label{eq7}
\gamma_j = i^{-j}\, \pi^{\frac{n}{2}}
\frac{\Gamma(\frac{j}{2})}{\Gamma ( \frac{n+j}{2})}\,.
\end{equation}
The quantities $\gamma_j$ appear in the computation of the Fourier
multiplier of the higher order Riesz transform $R$ with kernel given
by a homogeneous harmonic polynomial~$P$ of degree~$j$. One has (see
\cite[p.~73]{St})
$$
\widehat{Rf}(\xi)= \gamma_j\frac{P(\xi)}{|\xi|^j} \,
\hat{f}(\xi),\quad f \in L^2(\Rn)\,.
$$
As we will show later, the series $\sum_{j=j_o}^\infty
\gamma_{2j}\,Q_{2j-2j_0}(x)$ is convergent in $C^\infty(S^{n-1})$.
If $U$~is defined as the operator in the algebra~$A$ whose Fourier
multiplier is~$\gamma_{2j_0}^{-1}$~times the sum of the preceding
series, then $T=R \circ U$ for the higher order Riesz transform~$R$
given by the polynomial~$P_{2j_0}$. This shows that $(iii)$ follows
from $(iv)$.

To show that $(iii)$ implies $(iv)$ we prove that $P=
\lambda\,P_{2j_0}$ for some real number $\lambda \neq 0$. Since $P$
divides $P_{2j_0}$ by assumption, we only need to show that the
degree~$d$ of~$P$ must be~$2 j_0$. Now, let $\mu(\xi)$ denote the
Fourier multiplier of $U$, so that $\mu$ is a smooth function with
no zeros on the sphere. The Fourier multiplier of $T$ is
$$
\sum_{j=j_0}^\infty \gamma_{2j}\, P_{2j}(\xi) =
\gamma_{d}\,P(\xi)\,\mu(\xi),\quad \xi \in S^{n-1}\,.
$$
If $d$ is less than $2j_0$, then $P$ is orthogonal to all $P_{2j}$
\cite[p.~69]{St} and so
$$
\int_{|\xi|=1} P(\xi)^2 \,\mu(\xi)\,d\xi = 0\,,
$$
which yields $P(\xi)=0$, $|\xi|=1$, a contradiction.

\begin{remark}
Condition $(iii)$ is rather easy to check in practice. For instance,
take $n=2$ and consider the polynomial of fourth degree
$$
P(x,y)= xy+ x^4+y^4-6\, x^2y^2\,.
$$
The polynomial operator associated to~$P$ does not satisfy $(i)$ nor
$(ii)$, because the definition of $P$ above is also the spherical
harmonics expansion of $P$ and $xy$ clearly does not divide $
x^4+y^4-6\,x^2 y^2$. Section~7 contains other examples of polynomial
operators that do not satisfy $(i)$ nor $(ii)$.

On the other hand, the polynomial operator associated with
$$
P(x,y)= xy+ x^3 y-y^3 x\,,
$$
does satisfy $(i)$ and $(ii)$, but this is not the case for the
operator determined by
$$
P(x,y)= xy+ 2 (x^3 y-y^3 x)\,,
$$
although $xy$ obviously divides $x^3 y-y^3 x$. See section~8 for the
details.

Thus the condition on $\Omega$ so that $T$ satisfies $(i)$ or $(ii)$
is rather subtle.
\end{remark}

Having clarified the statement of the Theorem and some of its
implications, we say now a few words on the proofs and the
organization of the paper.

We devote sections 2, 3 and 4 to the proof of ``$(iii)$ implies
$(i)$", which we call the sufficient condition. In section~2 we
prove that the even higher order Riesz transforms satisfy $(i)$.
Section~3 is devoted to the proof of the sufficient condition for
polynomial operators. The argument is an extension of that used in
the previous section. The drawback is that we loose control on the
dependence of the constants on the degree of the polynomial. The
main difficulty we have to overcome in section~4 to complete the
proof of the sufficient condition in the general case, is to find a
second approach to the polynomial case which gives some estimates
with constants independent of the degree of the polynomial. This
allows us to use a compactness argument to finish the proof. It is
an intriguing fact that the approach in section~3 cannot be
dispensed with, because it provides certain properties which are
vital for the final argument and do not follow otherwise.

In sections~5 and 6 we prove the necessary condition, that is,
``$(ii)$ implies $(iii)$''. In section 5 we deal with the polynomial
case. Analysing the inequality $(ii)$ via Plancherel at the
frequency side we obtain various inclusion relations among zero sets
of certain polynomials. This requires a considerable combinatorial
effort for reasons that will become clear later on. We found in
Maple a formidable ally in formulating the right identities which
were needed, which we proved rigorously afterwards. In a second step
we solve the division problem which leads us to $(iii)$ by a
recurrent argument with some algebraic geometry ingredients, the
Hilbert's Nullstellensatz in particular. The question of
independence on the degree of the polynomial appears again, this
time related to the coefficients of certain expansions. We deal with
this problem in section~6. Section 7 is devoted to the proof of the
intricate combinatorial lemmas used in the previous sections. In
section 8 we discuss some examples and we ask a couple of questions
that we have not been able to answer.

Our methods are a combination of classical Fourier analysis
techniques and Calder\'{o}n-Zygmund theory with potential theoretic
ideas coming from our previous work \cite{MO}, \cite{MPV},
\cite{MV1}, \cite{MNOV}, \cite{Ve1} and \cite {Ve2}.

As it was discovered in \cite{MV2}, there is remarkable difference
between the odd and even cases for the problem we consider. To keep
at a reasonable size the length of this article we decided to only
deal here with the even case, which is more difficult, because one
needs special $L^\infty$ estimates for singular integrals, which
hold only in the even case. The results for the odd case will be
published elsewhere~\cite{MOPV}.

The $L^\infty$ estimates mentioned above are not obvious even for
the simplest even homogeneous Calder\'{o}n-Zygmund operator, the
Beurling transform, which plays an important role in planar
quasiconformal mapping theory. An application of our estimates to
planar quasiconformal mappings is given in \cite{MOV}.

\section{Even higher order Riesz transforms}

In this section we prove that if $T$ is an even higher order Riesz
transform, then
\begin{equation}\label{eq8}
T^{\star}f(x) \leq C \, M(Tf)(x),\quad x \in \Rn\,.
\end{equation}

Let $B$ be the open ball of center~$0$ and radius~$1$, $\partial B$
its boundary and $\overline B$ its closure. In proving \eqref{eq8}
we will encounter the following situation. We are given a
function~$\varphi$ defined by different formulae in $B$ and $\Rn
\setminus \overline B$, which is differentiable up to order~$N$ on
$B \cup (\BC)$ and whose derivatives up to order $N-1$ extend
continuously up to $\partial B$. The question is to compare the
distributional derivatives of order $N$ with the expressions one
gets on $B$ and $\BC$ by taking ordinary derivatives. The next
simple lemma is a sample of what we need.

\begin{lemma}\label{L1}
Let $\varphi$ be a continuously differentiable function on $B \cup
(\BC) $ which extends continuously to $\partial B$. Then we have the
identity
$$
\partial_j \varphi = \partial_j \varphi(x) \chi_B(x) +   \partial_j
\varphi(x)\chi_{\BC}(x)\,,
$$
where the left hand side is the $j$-th distributional derivative of
$\varphi$.
\end{lemma}

\begin{proof}
Let $\psi$ be a test function. Then
$$
\langle \partial_j \varphi, \psi \rangle = - \int \varphi\,
\partial_j \psi = - \int_B \varphi \,
\partial_j \psi  - \int_{\BC} \varphi \,
\partial_j \psi \,.
$$
Now apply Green-Stokes' theorem to the domains $B$ and $\BC$ to move
the derivatives from $\psi$ to $\varphi$. The boundary terms cancel
precisely because of the continuity of $\varphi$ on $\partial B$,
and we get
$$
\langle \partial_j \varphi, \psi \rangle =  \int (\chi_B \,
\partial_j \varphi + \chi_{\BC} \,
\partial_j \varphi ) \psi \,dx\,.
$$
\end{proof}

We need an analog of the previous statement for second order
derivatives and radial functions, which is the case we take up in
the next corollary.

\begin{co} \label{C1}
Assume that $\varphi$ is a radial function of the form
$$
\varphi(x)=   \varphi_1(|x|) \chi_B(x) +  \varphi_2(|x|)
\chi_{\BC}(x)\,,
$$
where $\varphi_1$ is continuously differentiable on $[0,1)$ and
$\varphi_2$ on $(1,\infty)$. Let $L$ be a second order differential
operator with constant coefficients. Then the distribution $L
\varphi$ satisfies
$$
L\varphi= L\varphi(x) \chi_B(x) + L\varphi(x) \chi_{\BC}(x)\,,
$$
provided $\varphi_1$, $\varphi_1'$, $\varphi_2$ and $\varphi_2'$
extend continuously to the point~$1$ and the two conditions
$$
\varphi_1(1)=  \varphi_2(1), \quad \varphi_1'(1) = \varphi_2'(1)\,,
$$
are satisfied.
\end{co}

\begin{proof}
The proof reduces to applying Lemma~\ref{L1} twice. Before the
second application one should remark that the hypothesis $
\varphi_1'(1) = \varphi_2'(1)\,$ gives the continuity of all first
order partial derivatives of $\varphi$.
\end{proof}

We proceed now to describe in detail the main argument for the proof
of  \eqref{eq8}. By translating and dilating one reduces the proof
of \eqref{eq8} to
\begin{equation}\label{eq9}
|T^1 f(0)| \leq C \, M(Tf)(0)\,,
\end{equation}
where
$$
T^1 f(0)= \int_{| y| > 1} f(y) K(y) \,dy
$$
is the truncated integral at level $1$. Recall that the kernel of
our singular integral is
$$
K(x)= \frac{\Omega(x)}{|x|^{n}}= \frac{P(x)}{|x|^{n+d}}\,,
$$
where $P$ is an even homogeneous harmonic polynomial of degree
$d\geq 2$. The idea is to obtain an identity of the form
\begin{equation}\label{eq10}
K(x)\chi_{\BC}(x)=  T(b)(x)\,,
\end{equation}
for some measurable bounded function $b$ supported on $B$. Once
\eqref{eq10} is at our disposition  we get, for $f$ in some $
L^p(\Rn)$, $1\leq p <\infty$,
\begin{equation*}
\begin{split}
T^1 f(0) & = \int \chi_{\BC}(y)\,K(y)\, f(y) \,dy \\
 & = \int T(b)(y)\,f(y)\,dy \\
 &= \int_B b(y) \, Tf(y)\,dy\,,
\end{split}
\end{equation*}
and so \eqref{eq9} follows with $C= V_n \, \|b \|_\infty$, $V_n$
being the volume of the unit ball of $\Rn$.

Let us turn our attention to the proof of \eqref{eq10}. Set $d=2N$
and let $E$ be the standard fundamental solution of the $N$-th power
$\triangle^N$ of the Laplacean. Consider the function
\begin{equation}\label{eq11}
\varphi(x)= E(x)\,\chi_{\BC}(x) + (A_0+A_1\,|x|^2 +\dotsb+
A_{d-1}\,|x|^{2d-2})\,\chi_B(x)\,,
\end{equation}
where the constants $A_0, A_1,\dotsc, A_{d-1}$ are chosen as
follows. Since $\varphi(x)$ is radial, the same is true of
$\triangle^j \varphi$ for each positive integer~$j$. Thus, in order
to apply the preceding corollary $N$~times one needs
$2N=d$~conditions, which (uniquely) determine $A_0, A_1,\dotsc,
A_{d-1}$. Therefore, for some constants $\alpha_1, \alpha_2,\dotsc,
\alpha_{N-1}$,
\begin{equation}\label{eq11bis}
\triangle^N \,\varphi = (\alpha_0+ \alpha_1 |x|^2 +\dotsb+
\alpha_{N-1} |x|^{2(N-1)})\chi_B(x) = b(x)\,,
\end{equation}
where the last identity is a definition of $b$. Since
$$
\varphi = E \star \triangle^N\,\varphi \,,
$$
taking derivatives of both sides we obtain
\begin{equation}\label{eq12}
P(\partial)\,\varphi = P(\partial)\, E \star \triangle^N \,
\varphi\,.
\end{equation}
To compute $P(\partial)E$ we take the Fourier transform
$$
\widehat{P(\partial)E}(\xi)= P(i\xi)\, \hat{E}(\xi)=
\frac{P(\xi)}{|\xi|^d}\,.
$$
On the other hand, as it is well known (\cite[p.~73]{St},
$$
\widehat{P.V.\frac{P(x)}{|x|^{n+d}}}\,(\xi)= \gamma_d
\frac{P(\xi)}{|\xi|^d}\,.
$$
See \eqref{eq7} for the precise value of $\gamma_d$, which is not
important now. We conclude that, for some constant $c_d$ depending
on $d$,
$$
P(\partial)E = c_d \,P.V.\frac{P(x)}{|x|^{n+d}}\,.
$$
Thus
$$
P(\partial)\varphi =
c_d\,P.V.\frac{P(x)}{|x|^{n+d}}\star\,\triangle^N\,\varphi =
 c_d\,T(b)\,.
$$

The only thing left is the computation of $P(\partial)\,\varphi$. We
have, by Corollary~\ref{C1},
$$
P(\partial)\,\varphi = c_d\, K(x)\,\chi_{\BC} +
P(\partial)(A_0+A_1\,|x|^2+\dotsb+
A_{d-1}\,|x|^{2d-2})(x)\,\chi_B(x)\,,
$$
and so, to complete the proof of \eqref{eq10}, we only have to show
that
\begin{equation}\label{eq13}
P(\partial)(|x|^{2j})= 0,\quad 1 \leq j \leq d-1\,.
\end{equation}
Notice that the degree of $P$ may be much smaller than the degree of
$|x|^{2j}$ and so the previous identity is not obvious. Taking the
Fourier transform we obtain
$$
\widehat{P(\partial)(|x|^{2j})}= c_j \,P(\xi)\,\triangle^j
\,\delta\,,
$$
where $\delta$ is the Dirac delta at the origin and $c_j$ a constant
depending on $j$. Let $\psi$ be a test function. Then, since $P$ is
harmonic,
\begin{equation*}
\begin{split}
\langle P(\xi)\,\triangle^j \,\delta, \varphi \rangle &= \langle
\triangle^j \,\delta ,\, P(\xi)\, \varphi(\xi)\rangle \\
& =\langle \triangle^{j-1} \,\delta, \,2\, \nabla P(\xi)\cdot \nabla
\varphi(\xi)+ P(\xi)\,\triangle\varphi(\xi) \rangle\,.
\end{split}
\end{equation*}
Iterating the previous computation we obtain that
$$
\langle P(\xi)\,\triangle^j \,\delta,\, \varphi \rangle = \langle
\delta, D(\xi)\rangle= D(0)\,,
$$
where $D$ is a linear combination of products of the form
$\partial^\alpha \,\varphi(\xi)\,\partial^\beta\,P(\xi)$, with
multi-indeces $\beta$ of length $|\beta| \leq j \leq d-1$. Therefore
$\partial^\beta\,P(\xi)$ is a homogeneous polynomial of degree at
least $d-j \geq 1$, and so $\partial^\beta\,P(0)=0$. This yields
$D(0)=0$ and completes the proof of \eqref{eq13} and, thus, of
\eqref{eq10}.

In fact \eqref{eq13} follows immediately from an identity of Lyons
and Zumbrun \cite{LZ} which will be discussed in the next section,
but we prefer to present here the above independent natural argument
for the reader's convenience.

\section{Proof of the sufficient condition: the polynomial case}

In this section we assume that $T$ is an even  polynomial operator.
This amounts to say that for some even integer $2N$, $N \geq 1$, the
function $ |x|^{2N}\,\Omega(x)$ is a homogeneous polynomial of
degree~$2N$. Such a polynomial may be written as \cite[p.~69]{St}
$$
|x|^{2N}\,\Omega(x)=  P_2(x)|x|^{2N-2}+ \dotsb+
P_{2j}(x)|x|^{2N-2j}+\dotsb+ P_{2N}(x)\,,
$$
where $P_{2j}$ is a homogeneous harmonic polynomial of degree~$2j$,
$1\leq j \leq N $. In other words, the expansion of~$\Omega(x)$ in
spherical harmonics is
$$
\Omega(x)= P_2(x)+P_4(x)+\dotsb+P_{2N}(x),\quad |x|=1\,.
$$

As in the previous section, we want to obtain an expression for the
kernel $K(x)$ off the unit ball $B$. For this we need the
differential operator $Q(\partial)$ defined by the polynomial
$$
 Q(x)=  \gamma_2 \, P_2(x)|x|^{2N-2}+\dotsb+ \gamma_{2j}\,
P_{2j}(x)|x|^{2N-2j}+\dotsb+ \gamma_{2N}\,P_{2N}(x) \,.
$$
If $E$ is the standard fundamental solution of $\Delta^N$, then
$$
Q(\partial)E = P.V. \,K(x)\,,
$$
which may be easily verified by taking the Fourier $E$ transform of
both sides.

Take now the function $\varphi$ of the previous section. We have
$\varphi = E \star \triangle^N\,\varphi$ and thus
$$
Q(\partial)\varphi = Q(\partial)E \star  \triangle^N\,\varphi =
P.V.\, K(x) \star b = T(b)\,,
$$
where $b$ is defined as $ \triangle^N\,\varphi$. On the other hand,
by Corollary~2
\begin{equation}\label{eq14}
Q(\partial)\,\varphi =  K(x)\,\chi_{\BC} +
Q(\partial)(A_0+A_1\,|x|^2+\dotsb+
A_{2N-1}\,|x|^{4N-2})(x)\,\chi_B(x)\,.
\end{equation}
Contrary to what happened in the previous section, the term
$$
S(x) : = - Q(\partial)(A_0+A_1\,|x|^2+\dotsb+
A_{2N-1}\,|x|^{4N-2})(x)
$$
does not necessarily vanish, the reason being that now $Q$ does not
need to be harmonic.

Our goal is to find a function $\beta \in L^\infty(\Rn)$, satisfying
the decay estimate
\begin{equation}\label{eq15}
|\beta(x)| \leq \frac{C}{|x|^{n+1}},\quad |x|\geq 2\,,
\end{equation}
and
\begin{equation}\label{eq16}
S(x) \chi_B(x) = T(\beta)(x)\,.
\end{equation}
Once this is achieved the proof of $(i)$ is just a variation of the
argument presented in section 2, which we now explain. By
\eqref{eq14}, the definition of $S(x)$ and \eqref{eq16}, we get
\begin{equation}\label{eq17}
K(x)\chi_{\BC}(x)=  T(b)(x)+ T(\beta)(x)\,.
\end{equation}
Set $ \gamma = b+ \beta$. We show \eqref{eq9} by arguing as follows.
For $f$ in any $ L^p(\Rn)$, $1\leq p <\infty$, we have
\begin{equation*}
\begin{split}
T^1 f(0) & = \int \chi_{\BC}(y)\,K(y)\, f(y) \,dy \\
 & = \int T(\gamma)(y)\,f(y)\,dy \\
 & = \int \gamma(y)\, Tf(y)\,dy \\
 &= \int_{2B} \gamma(y) \, Tf(y)\,dy + \int_{\Rn \setminus 2B} \gamma(y) \,
 Tf(y)\,dy
\end{split}
\end{equation*}
and thus, by the decay inequality \eqref{eq15} with $\beta$ replaced
by $\gamma$,
\begin{equation*}
\begin{split}
|T^1 f(0)| & \leq C  \left( \|\gamma\|_\infty
\frac{1}{|2B|}\int_{2B}|Tf(y)|\,dy
 + \int_{\Rn \setminus 2B} \frac{|Tf(y)|}{|y|^{n+1}}\,dy \right)\\*[5pt]
 & \leq C\,M(Tf)(0)\,.
\end{split}
\end{equation*}

To construct $\beta$ satisfying  \eqref{eq15} and \eqref{eq16} we
resort to our hypothesis, condition~$(iii)$ in the Theorem, which
says that $T = R \circ U$, where $R$ is a higher order Riesz
transform, $U$ is an invertible operator in the algebra $A$ and the
polynomial $P$ which determines $R$ divides $P_{2j}$, $1\leq j \leq
N $, in the ring of polynomials in $n$~variables with real
coefficients. The construction of $\beta$ is performed in two steps.

The first step consists in proving that there exists a
function~$\beta_1$ in $L^\infty(B)$, satisfying a Lipschitz
condition of order~$1$ on $B$,  $\int \beta_1(x)\,dx =0$ and such
that
\begin{equation}\label{eq18}
S(x) \chi_B(x) = R(\beta_1)(x)\,.
\end{equation}
It will become clear later on how the Lipschitz condition on
$\beta_1$ is used. To prove \eqref{eq18} we need an explicit formula
for $S(x)$ and for that we will make use of the following formula of
Lyons and Zumbrun \cite{LZ}.

\begin{lemma} \label{L2}
Let $L$ be a homogeneous polynomial of degree $l$ and let $f$ be a
smooth function of one variable. Then
$$
L(\partial)f(r)= \sum_{\nu \geq 0} \frac{1}{2^\nu\,\nu !}\,
\Delta^\nu L(x)\left(\frac{1}{r}\frac{\partial}{\partial
r}\right)^{l-\nu} f(r), \quad r=|x|\,.
$$
\end{lemma}

An immediate consequence of Lemma~3 is

\begin{lemma}\label{L3}
Let $P_{2j}$ a homogeneous harmonic polynomial of degree $2j$ and
let $k$ be a non-negative integer. Then
$$
P_{2j}(\partial)(|x|^{2k}) = 2^{2j}\frac{k!}{(k-2j)!}
\,P_{2j}(x)\,|x|^{2(k-2j)}\quad \text{if }\, 2j\leq k\,,
$$
and
$$
P_{2j}(\partial)(|x|^{2k}) = 0,\quad \text{if }\, 2j > k\,.
$$
\end{lemma}

On the other hand, a routine computation gives
\begin{equation}\label{eq19}
\triangle^{j}(|x|^{2k}) = 4^j\frac{j! \,k!
}{(k-j)!}\binom{\frac{n}{2}+k-1}{j} \,|x|^{2(k-j)},\quad k \geq j\,,
\end{equation}
and
\begin{equation}\label{eq20}
\triangle^{j}(|x|^{2k}) = 0 ,\quad k < j \,.
\end{equation}
By Lemma 4, \eqref{eq19} and \eqref{eq20} we get that for some
constants $c_{jk}$ one has, in view of the definitions of $Q(x)$ and
$S(x)$,
\begin{equation}\label{eq21}
S(x)= \sum_{j=1}^{N-1} \sum_{k=j}^{N-1}
c_{jk}\,P_{2j}(x)\,|x|^{2(k-j)}\,.
\end{equation}
Therefore it suffices to prove \eqref{eq18} with $S(x)$ replaced by
$P_{2j}(x)\,|x|^{2k}$, for $1 \leq j \leq N$ and each non-negative
integer~$k$. The idea is to look for an appropriate function~$\psi$
such that
\begin{equation}\label{eq22}
P(\partial)\psi (x) = P_{2j}(x)\,|x|^{2k}\,\chi_B(x)\,.
\end{equation}
Indeed, if \eqref{eq22} holds and $2d$ is the degree of $P$, then
$$
\psi = E \star \Delta^d \psi\,,
$$
provided $E$~is the fundamental solution of $\triangle^d$ and $\psi$
is good
 enough. Hence
 $$
P(\partial)\psi = P(\partial)E \star \Delta^d \psi = c\,
P.V.\,\frac{P(x)}{|x|^{n+2d}} \star \Delta^d \psi= R(\beta_1)\,,
$$
if $\beta_1 = c\, \Delta^d \psi$. The conclusion is that we have to
solve \eqref{eq22} in such a way that $\Delta^d \psi$~is supported
on $B$, is a Lipschitz function on $B$ and has zero integral.

Taking Fourier transforms in \eqref{eq22} we get
\begin{equation}\label{eq23}
(-1)^d P(\xi)\, \widehat{\psi} (\xi) =
(-1)^{j+k}\,P_{2j}(\partial)\,\triangle^k
\left(\widehat{\chi_B}(\xi)\right)\,.
\end{equation}
Recall that for $m=n/ 2$ one has \cite[A-10]{Gr}
$$
\widehat{\chi_B}(\xi)= \frac{J_m(\xi)}{|\xi|^m},\quad \xi \in \Rn\,,
$$
where $J_m$ is the Bessel function of order~$m$. Set
$$
G_\lambda(\xi)= \frac{J_\lambda(\xi)}{|\xi|^\lambda},\quad \xi \in
\Rn, \quad \lambda>0\,.
$$
In computing the right hand side of \eqref{eq23} we apply Lemma~3 to
$L(x)= P_{2j}(x) \,|x|^{2k}$ and $f(r)= G_m(r)$ and we get
$$
P(\xi)\, \widehat{\psi} (\xi) = (-1)^{j+k+d} \sum_{\nu \geq 0}
\frac{(-1)^\nu}{2^\nu\,\nu !}\,\triangle^\nu
\left(P_{2j}(\xi)\,|\xi|^{2k}\right)\,G_{m+2j+2k-\nu}(\xi)\,,
$$
owing to the well known formula, e.g.~\cite[A-6]{Gr},
$$
\frac{1}{r}\,\frac{d}{dr}\,G_{\lambda}(r) = -
G_{\lambda+1}(r)\,,\quad r>0,\quad \lambda>0\,.
$$
Since $P_{2j}(\xi)$ is homogeneous of degree $2j$, $\nabla
P_{2j}(\xi)\cdot \xi = 2j\,P_{2j}(\xi)$, and hence one may readily
show by an inductive argument that
$$
 \triangle^\nu \left(P_{2j}(\xi)\,|\xi|^{2k}\right)=
 a_{jk\nu}\,P_{2j}(\xi)\,|\xi|^{2(k-\nu)}\,,
$$
for some constants $a_{jk\nu}$. Thus, for some other constants $
a_{jk\nu}$, we get
\begin{equation}\label{eq24}
P(\xi)\, \widehat{\psi} (\xi) = \sum_{\nu \geq 0}
a_{jk\nu}\,P_{2j}(\xi)\,|\xi|^{2(k-\nu)}\,G_{m+2j+2k-\nu}(\xi)\,.
\end{equation}
By hypothesis $P$ divides $P_{2j}$ in the ring of polynomials in $n$
variables and so
$$
P_{2j}(\xi)= P(\xi)\,Q_{2j-2d}(\xi)\,,
$$
for some homogeneous polynomial $Q_{2j-2d}$ of degree $2j-2d$.
Cancelling out the factor $P(\xi)$ in \eqref{eq24} we conclude that
$$
 \widehat{\psi} (\xi) = Q_{2j-2d}(\xi)\sum_{\nu = 0}^k
a_{jk\nu}\,|\xi|^{2(k-\nu)}\,G_{m+2j+2k-\nu}(\xi)\,.
$$
Since \cite[A-10]{Gr}
$$
\widehat{\left((1-|x|^2)^\lambda \,\chi_B(x)\right)}(\xi) =
c_\lambda\,
 G_{m+\lambda}(\xi)\,,
$$
we finally obtain
$$
 \psi(x) = Q_{2j-2d}(\partial)\sum_{\nu = 0}^k
a_{jk\nu}\,\triangle^{k-\nu}\left( (1-|x|^2)
^{2j+2k-\nu}\,\chi_B(x)\right)\,.
$$
Observe that $\psi$ restricted to $B$ is a polynomial which vanishes
 on $\partial B$ up to order~$2d$ and $\psi$ is zero off $B$. Therefore $
\triangle^d \psi$ is supported on~$B$ and its restriction to~$B$ is
a polynomial with zero integral. This completes the first step of
the construction of~$\beta$.

The second step proceeds as follows. Since by hypothesis $T=R\circ U
$, with $U$~invertible in the algebra~$A$, we have
$$
R(\beta_1)=  T(U^{-1}\beta_1)\,.
$$
Setting
\begin{equation}\label{eq24bis}
\beta=U^{-1}\beta_1\,,
\end{equation}
we are only left with the task of showing that
\begin{equation}\label{eq25}
\beta \in L^\infty(\Rn)
\end{equation}
and that, for some positive constant $C$,
\begin{equation}\label{eq26}
 | \beta (x)|  \leq \frac{C}{|x|^{n+1}},\quad |x| \geq 2\,.
\end{equation}
Since $U^{-1} \in A$\,, for some real number $\lambda$ and some
smooth homogeneous Calder\'{o}n-Zygmund operator $V$,
$$
U^{-1} = \lambda\,I + V \,.
$$
Thus
$$
\beta = \lambda\,\beta_1 + V (\beta_1)\,.
$$
Now $\beta_1$ is supported on $B$ and has zero integral on $B$ and
this is enough to insure the decay estimate \eqref{eq26}. Indeed,
let $L(x)$ be the kernel of $V$ and assume that $|x| \geq 2$. Then
\begin{equation}\label{eq27}
\begin{split}
V(\beta_1)(x) &=  \int L(x-y) \,\beta_1(y)\,dy \\
              &=  \int \left(L(x-y) - L(x)
              \right)\,\beta_1(y)\,dy\,,
\end{split}
\end{equation}
and so
\begin{equation}\label{eq28}
\begin{split}
|V(\beta_1)(x)| &  \leq \int |\left(L(x-y) - L(x)
              \right)|\,|\beta_1(y)|\,dy,\\
                & \leq C\, \int  \frac{|y|}{|x|^{n+1}} \,|\beta_1(y)|\,dy,\\
                & = \frac{C}{|x|^{n+1}}\,.
\end{split}
\end{equation}

The boundedness of $\beta$ is a more delicate issue. It follows
immediately from the next lemma applied to the operator~$V$ and the
function $\beta_1$. Is precisely here where we use the fact that
$\beta_1$ satisfies a Lipschitz condition.

The constant of the kernel $K(x)= \Omega(x)/|x|^n$ of the smooth
homogeneous Cal\-de\-r\'{o}n-Zygmund operator $T$ is
\begin{equation}\label{eq29}
\|T\|_{CZ} \equiv \|K\|_{CZ}= \|\Omega \|_\infty+ \| |x|\, \nabla
\Omega(x) \|_\infty\,.
\end{equation}

We adopt the standard notation for the minimal Lipschitz constant of
a Lipschitz function~$f$ on~$B$, namely
$$
\|f\|_{\operatorname{Lip}(1,B)} = \sup\left \{
\frac{|f(x)-f(y)|}{|x-y|}: x,y \in B, \, x \neq y\right \} <
\infty\,.
$$

\begin{lemma}\label{Lip}
Let $T$ be the homogeneous singular integral operator with kernel
$K(x)= \frac{\Omega(x)}{|x|^n}$, where $\Omega$ is an even
homogeneous function of degree $0$, continuously differentiable and
with zero integral on the unit sphere. Then
$$
\|T(f\,\chi_B)\|_{L^\infty(\Rn)} \leq C \, \|K\|_{CZ} \left(
\|f\|_{L^\infty(B)} + \|f\|_{\operatorname{Lip}(1,B)}\right)\,,
$$
where $C$ is a positive constant which depends only on $n$.
\end{lemma}

\begin{proof}
We start by examining the behaviour of $T(f\,\chi_B)$ on the unit
sphere. We claim that
$$
|T_\epsilon(f\,\chi_B)(a)| \leq C\, \|K\|_{CZ} \left(
\|f\|_{L^\infty(B)} + \|f\|_{\operatorname{Lip}(1,B)}\right),\quad
|a|=1,\quad \epsilon>0\,.
$$
Indeed, if one follows in detail the proof of the claim, which we
discuss below, one will realize that the principal value integral
$T(f\,\chi_B)(a)$ exists for all $a$ in the sphere and satisfies the
desired estimate.

We have
\begin{equation*}
\begin{split}
T_\epsilon(f\,\chi_B)(a)&= \int_{\epsilon < |x-a|< 1/2}
\chi_B(x)\,f(x)\, K(a-x)\,dx + \int_{1/2 < |x-a|} \dotsi\\*[5pt] &=
I_\epsilon + II \,.
\end{split}
\end{equation*}
Clearly,
$$
|II| \leq  \int_{1/2 < |x-a|} \chi_B(x)\,|f(x)|\,
\frac{|\Omega(x-a)|}{|x-a|^n}\,dx \leq 2^n\,|B|\,\|\Omega\|_\infty\,
\|f\|_{L^\infty(B)} \,.
$$
To deal with the term $I_\epsilon$ we write
\begin{equation*}
\begin{split}
I_\epsilon &= \int_{\epsilon < |x-a|< 1/2} \chi_B(x)\,(f(x)-f(a))\,
K(a-x)\,dx\\*[5pt] &\quad+ f(a)\int_{\epsilon < |x-a|< 1/2}
\chi_B(x)\, K(a-x)\,dx \\*[5pt] & = III_\epsilon+ f(a)\,IV_\epsilon
\,,
\end{split}
\end{equation*}
and we remark that $III_\epsilon$ can easily be estimated as follows
$$
|III_\epsilon| \leq \|f\|_{\operatorname{Lip}(1,B)}\,\int_B |x-a|
|K(a-x)|\,dx \leq C\,\|\Omega\|_\infty\,\|f
\|_{\operatorname{Lip}(1,B)}\,.
$$
Taking care of $IV_\epsilon$ is not so easy. Take spherical
coordinates centered at the point $a$, $x=a+r\, \omega$ with $0 \leq
r$ and $|\omega|=1$. Then
\begin{equation}\label{eq30}
IV_\epsilon  =  \int_\epsilon^{1/2} \left( \int_{A(r)}
\Omega(\omega)\,d\sigma(\omega)\right)\frac{dr}{r} \,,
\end{equation}
where
$$
A(r)= \{ \omega : |\omega| =1 \text{ and }|a+r \omega| < 1 \}\,.
$$
Let $H$ be the tangent hiperplane to $S= \{\omega : |\omega|=1\}$ at
the point~$a$. Call $V$ the half space with boundary $H$ containing
the origin. Clearly $A(r) \subset S \cap V$. Since $\Omega$~is even,
$$
0= \int_S \Omega(\omega)\,d\sigma(\omega) = 2\, \int_{S \cap V}
\Omega(\omega)\,d\sigma(\omega)\,.
$$
Thus
$$
\int_{A(r)} \Omega(\omega)\,d\sigma(\omega) = - \int_{(S \cap V)
\setminus A(r)} \Omega(\omega)\,d\sigma(\omega)\,,
$$
and so
$$
\left|\int_{A(r)} \Omega(\omega)\,d\sigma(\omega)\right| \leq \|
\Omega\|_\infty \, \sigma((S \cap V) \setminus A(r))\,.
$$
Since $H$ is tangent to $S$ at the point $a$, we obtain
\begin{equation*}
\sigma((S \cap V) \setminus A(r)) \leq C\,r\,,
\end{equation*}
which yields, by \eqref{eq30},

\begin{equation*}
|IV_\epsilon| \leq C\,\|\Omega \|_\infty\,.
\end{equation*}

\vspace*{3pt}

\begin{center}
\begin{figure}[ht]
\epsfig{file=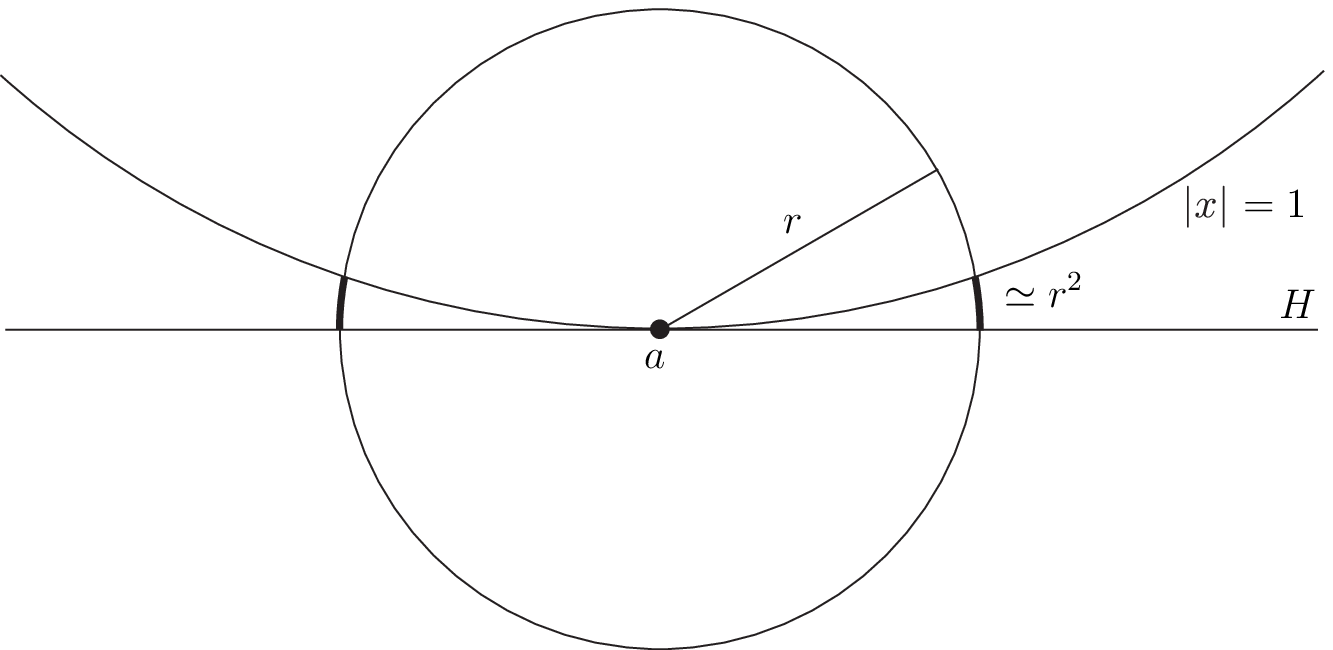} \label{lhoritzo}
\end{figure}
\end{center}

Assume now that $|a| < 1$. Proceeding as before we estimate in the
same way the terms $II$ and $III_\epsilon$, so that we are again
left with $IV_\epsilon$. Let $\epsilon_0$ stand for the distance
from $a$ to the boundary of~$B$. In estimating $IV_\epsilon$ we can
assume, without loss of generality, that $\epsilon_0 \leq 1/4$. .Set
$a_0 = a/|a|$,
$$
A= \{x \in B : \epsilon_0 < |x-a| < 1/2 \}
$$
and
$$
A_0= \{x \in B : \epsilon_0 < |x-a_0| < 1/2 \}\,.
$$
We compare $IV_\epsilon$ to the expression we get replacing $a$ by
$a_0$ and $\epsilon$ by $\epsilon_0$ in the definition of
$IV_\epsilon$. For $\epsilon \leq \epsilon_0$ we have
$$
\int_{\epsilon < |x-a|< 1/2} \chi_B(x)\, K(a-x)\,dx =
\int_{\epsilon_0 < |x-a|< 1/2} \chi_B(x)\, K(a-x)\,dx
$$
and then
\begin{equation*}
\begin{split}
\biggl|\int_{\epsilon < |x-a|< 1/2} &\chi_B(x)\, K(a-x)\,dx -
\int_{\epsilon_0 < |x-a_0|< 1/2} \chi_B(x)\, K(a_0-x)\,dx \biggr|
\\*[7pt] &= \left|\int_{A}  K(a-x)\,dx - \int_{A_0} K(a_0-x)\,dx
\right|  \\*[7pt] &\leq \int_{A\cap A_0}
|K(a-x)-K(a_{0}-x)|\,dx\\*[7pt] &\quad+ \left|\int_{A \setminus A_0}
\!\chi_B(x)\, K(a-x)\,dx \right| +
 \left| \int_{A_0 \setminus A} \chi_B(x)\, K(a_0-x)\,dx  \right|  \\*[7pt]
&= J_1+J_2+J_3\,.
\end{split}
\end{equation*}
If $x \in A \cap A_0$, then
$$
|K(a-x)-K(a_0-x)| \leq C \,\|K\|_{CZ}\,
\frac{|a-a_0|}{|x-a|^{n+1}}\,.
$$
Hence
$$
J_1 \leq C \,\|K\|_{CZ}\, |a-a_0|\, \int_{|x-a|> \epsilon_0}
\frac{dx}{|x-a|^{n+1}} \leq C\,\|K\|_{CZ}\,.
$$
To estimate $J_2$ observe that
$$
A \setminus A_0 = \left(A \cap B(a_0,\epsilon_0)\right) \cup \left(A
\cap (\Rn \setminus B(a_0,1/2))\right)\,.
$$
Now, it is obvious that if $|x-a_0|\geq 1/2$, then $|x-a| \geq 1/4$,
and so
$$
J_2 \leq \|\Omega\|_\infty \left( \int_{|x-a_0|< \epsilon_0}
\frac{dx}{\epsilon_0^n} + \int_B 4^n \,dx \right) \leq
C\,\|\Omega\|_\infty \,.
$$
A similar argument does the job for $J_3$.

The case $|a|>1$ is treated in a completely analogous way.
\end{proof}

The construction of $\beta$ is then completed and the Theorem is
proved for polynomial operators.

We remark that a variant of Lemma~\ref{Lip} holds, with the same
proof,  replacing $B$ by $\BC$. To control the term $II$ we have to
assume, in addition to the hypothesis of Lemma~\ref{Lip}, that
 $f$ satisfies a decay inequality of the type
$$
|f(x)| \le \frac{\|f \|_{L^\infty(\BC)}}{|x|^\eta}, \quad |x| \geq
1\,.
$$
Then we conclude that
$$
\|T(f\,\chi_{\BC})\|_{L^\infty(\Rn)} \leq C \, \|K\|_{CZ} \left(
\|f\|_{L^\infty(\BC)} + \|f\|_{\operatorname{Lip}(1,\BC)}\right)\,,
$$
where $C$ depends on $n$ and $\eta$. We will use later on this
variant of Lemma~\ref{Lip} with $f(x)= K(x)$ on $\BC$, so that
$\eta=n$ and the constant $C$ will depend only on $n$.

We mention another straightforward extension of Lemma~\ref{Lip} that
will not be used in this paper. The function $f$ may be assumed to
be in $\operatorname{Lip}(\alpha,B)$\,\, $0<\alpha \leq 1$\,, and
the unit ball may be replaced by a domain with boundary of class
$C^{1+\epsilon}$\,.

After the paper was completed we learned from Stephen Semmes that
Lemma~\ref{Lip} is known in dimension $2$ \cite[p.52]{Ch} or
\cite[p.348]{BM} and that was used to prove global regularity of
vortex patches for incompressible perfect fluids.

\section{Proof of the sufficient condition: the general case}
We start this section by clarifying several facts about the
convergence of the series~\eqref{eq6}. Let us then assume that
$\Omega$ is a function in $C^\infty(S^{n-1})$ with zero integral.
Then $\Omega$ has an expansion~\eqref{eq6} in spherical harmonics.
For each positive integer~$r$, one has the identity \cite[p.~70]{St}
\begin{equation}\label{eq30bis}
 \sum_{j\geq 1} \left(j(j+n-2)\right)^r\,\|P_j\|^2_2 =
(-1)^r\int_{S^{n-1}} \triangle^r_S \Omega \,\Omega \,d\sigma \,,
\end{equation}
where $\Delta_S$ stands for the spherical Laplacean. Then
$$
\sum_{j\geq 1} \left(j(j+n-2)\right)^r\,\|P_j\|^2_{2} \leq  \|
\triangle^r_S \Omega \|_2 \,\|\Omega  \|_2 \,,
$$
where the $L^2$ norm is taken with respect to $d\sigma$. Thus, by
Schwarz's inequality, for each positive integer~$M$
\begin{equation}\label{eq31}
\sum_{j\geq 1} j^M \,\|P_j\|_{2} < \infty\,.
\end{equation}
We want to see that we also have
\begin{equation}\label{eq32}
\sum_{j\geq 1} j^M \,\|P_j\|_{\infty} < \infty\,,
\end{equation}
where the supremum norm is taken on $S^{n-1}$. This follows
immediately from the next lemma, whose proof was indicated to us by
Fulvio Ricci.

\begin{lemma}\label{Hom}
For all homogeneous polynomials of degree $q$
$$
\| Q\|_\infty \leq C\,q^{\frac{n-1}{2}}\,\|Q \|_2\,,
$$
where $C$ is a positive constant which depends only on $n$.
\end{lemma}

\begin{proof}
Take an orthonormal base $Q_1,\dotsc,Q_d$, $d=d_q$, of the subspace
of $L^2(d\sigma)$ consisting of the restrictions to $S^{n-1}$ of all
homogeneous polynomials of degree~$q$. Consider the function
$$
S(x)=\sum_{j=1}^{d} Q_j(x)^2,\quad x \in S^{n-1}\,.
$$
We claim that $S$ is rotation invariant, and, hence, constant. Since
$d\sigma$ is a probability measure this constant must be $
\sum_{j=1}^{d} \int Q_j(x)^2\,d\sigma(x) = d$. Now let $Q$ be a
homogeneous polynomial of degree $q$ and set $Q= \sum_{j=1}^d
\lambda_j\,Q_j$. Then
$$
|Q(x)| \leq \left( \sum_{j=1}^d \lambda_j^2 \right)^{\frac{1}{2}}
\left( \sum_{j=1}^d Q_j(x)^2 \right)^{\frac{1}{2}} = \|Q
\|_2\,d^{\frac{1}{2}}, \quad x \in S^{n-1}\,,
$$
which proves the lemma because $d = \binom{n+q-1}{q} \simeq q^{n-1}$
(\cite [p.~139]{SW}).

To show the claim take a rotation $\rho$. Then  we have
$$
Q_j(\varrho(x))= \sum_{j=1}^d  a_{jk}\,Q_k(x)\,,
$$
for some matrix $(a_{jk})$ which is orthogonal, because the
polynomials $Q_j(\rho(x))$ form also an orthonormal basis due to the
rotation invariance of $\sigma$. Hence
$$
\sum_{j=1}^{d} Q_j(\varrho(x))^2  = \sum_{j=1}^{d} Q_j(x)^2 ,\quad x
\in S^{n-1}\,.
$$
\end{proof}

Let us return now to the context of the Theorem. Thus $T$ is an even
 smooth homogeneous Calder\'{o}n-Zygmund operator with kernel $K(x) = \Omega(x) /
|x|^n$, and the expansion of $\Omega$ in spherical harmonics is
\begin{equation}\label{eq33}
\Omega(x) =  \sum_{j \geq 1}^\infty \,P_{2j}(x)\,.
\end{equation}
By hypothesis there is a homogeneous harmonic polynomial $P$ of
degree $2d$ which divides each $P_{2j}$. In other words, $P_{2j} = P
\,Q_{2j-2d}$, where $Q_{2j-2d}$ is a homogeneous polynomial of
degree $2j-2d$. We want to show that the series $ \sum_{j}
Q_{2j-2d}(x)$ is convergent in $C^\infty(S^{n-1})$, that is, that
for each positive integer $M$
\begin{equation}\label{eq34}
\sum_{j\geq d} j^M \,\|Q_{2j-2d}\|_{\infty} < \infty\,.
\end{equation}
The next lemma states that when one divides two homogeneous
polynomials, then the supremum norm (on $S^{n-1}$) of the quotient
is controlled by the supremum norm of the dividend.

\begin{lemma}\label{Div}
Let $P$ be a homogeneous polynomial non identically zero. Then there
exists a positive $\epsilon$ and a positive constant $C=C(n,P)$ such
that
$$
\| Q\|_\infty \leq C \,q^{2(n-1)/ \epsilon}\,\|P\,Q \|_\infty\,,
$$
for each homogeneous polynomial $Q$ of degree $q$.
\end{lemma}

\begin{proof}
Assume that we can prove that for some positive $\epsilon$
\begin{equation}\label{eq35}
\int_{|x|=1} \frac{1}{|P(x)|^\epsilon} \,d\sigma(x) < \infty \,.
\end{equation}
Then, by Lemma~6 and Schwarz's inequality,
\begin{equation*}
\begin{split}
 \|Q\|_ \infty & \leq \,C \, q^{(n-1)/2}\, \|Q\|_2 \\*[5pt]
               & \leq C \, q^{(n-1)/2} \left(\int_{|x|=1} \frac{1}{|P(x)|^\epsilon} \,d\sigma(x)\right)^{1/4}
\left(\int_{|x|=1} |P(x)|^\epsilon \, |Q|^4 \,
d\sigma(x)\right)^{1/4} \\*[5pt] & \leq C \, q^{(n-1)/2}\,
\|P\,Q\|_\infty^{\epsilon/4}  \left(\int_{|x|=1} |Q|^{4-\epsilon} \,
d\sigma(x)\right)^{1/4} \\*[5pt]
 & \leq  C \, q^{(n-1)/2}\, \|P\,Q\|_\infty^{\epsilon/4} \,
 \|Q\|_\infty^{1-\epsilon/4}\,,
\end{split}
\end{equation*}
which completes the proof of the lemma.

Let us prove \eqref{eq35}. Let $d$ be the degree of $P$. By a
well-known result of Ricci and Stein \cite {RS}, $|P(x)|$ is a
weight in the class $A^\infty$. Indeed, if $\epsilon \,d < 1$, then
$$
\int_{ |x| < 1} \frac{1}{|P(x)|^\epsilon} \,dx \le C(\epsilon,d)
\left(\int_{|x| < 1} |P(x)| \,dx \right)^{-\epsilon} < \infty \,.
$$
Since $P$ is an homogeneous polynomial,  \eqref{eq35} follows by
changing to spherical coordinates.
\end{proof}

Now  \eqref{eq34} may be proved readily from Lemma~7
and~\eqref{eq32}. Indeed, setting $M_0 = 2(n-1)/\epsilon$, we have
$$
\|Q_{2j-2d}\|_{\infty} \leq \,C(n,P)\, (2j)^{M_0}
\,\|P_{2j}\|_{\infty}\,,
$$
and
$$
\sum_{j\geq d} j^M \,\|Q_{2j-2d}\|_{\infty}  \leq \,C(n,P)
\sum_{j\geq 1} (2j)^{M+M_0} \,\|P_{2j}\|_{\infty} < \infty\,.
$$

The scheme for the proof of the sufficient condition in the general
case is as follows.  Taking a large partial sum of the series
\eqref{eq33} we pass to a polynomial operator $T_N$ (associated to a
polynomial of degree $2N$), which still satisfies the
hypothesis~$(iii)$ of the Theorem. Then we may apply the
construction of section~3 to $T_N$ and get functions $b_N$ and
$\beta_N$. Unfortunately what was done in section~3 does not give
any uniform estimate  in $N$, which is precisely what we need to try
a compactness argument. The rest of the section is devoted to get
the appropriate uniform estimates and to describe the final
compactness argument.

By hypothesis,  $T=R\circ U$, where $R$ is the higher order Riesz
transform associated to the harmonic polynomial~$P$ of degree~$2d$
that divides all $P_{2j}$, and $U$ is invertible in the algebra~$A$.
The Fourier multiplier of $T$ is
$$
\sum_{j=d}^\infty \gamma_{2j}\, \frac{P_{2j}(\xi)}{|\xi|^{2j}} =
\gamma_{2d}\,\frac{P(\xi)}{|\xi|^{2d}}\,\sum_{j\geq d}
\frac{\gamma_{2j}}
{\gamma_{2d}}\,\frac{Q_{2j-2d}(\xi)}{|\xi|^{2j-2d}},\quad \xi \in
\Rn \setminus \{0\}\,.
$$
Therefore the Fourier multiplier of $U$ is
\begin{equation}\label{eq36}
\mu(\xi) =  \gamma_{2d}^{-1}\, \sum_{j\geq d} \gamma_{2j}
\,\frac{Q_{2j-2d}(\xi)}{|\xi|^{2j-2d}} \,,
\end{equation}
and the series is convergent in $C^\infty(S^{n-1})$ because
$\gamma_{2j}\simeq (2j)^{-n/2}$ \cite[p.~226]{SW}. Set, for $N\geq
d$,
\begin{equation}\label{eq37}
\mu_N(\xi) =  \gamma_{2d}^{-1}\, \sum_{j=d}^N \gamma_{2j}
\,\frac{Q_{2j-2d}(\xi)}{|\xi|^{2j-2d}},\quad \xi \in \Rn \setminus
\{0\}\,.
\end{equation}
If
$$
K_N (x) = \sum_{j=d}^N \frac{P_{2j}(x)}{|x|^{2j+n}},\quad x \in \Rn
\setminus \{0\}\,,
$$
and $T_N$ is the polynomial operator with kernel $K_N$, then $T_N =
R \circ U_N$, where $U_N$~is the operator in the algebra~$A$ with
Fourier multiplier $\mu_N(\xi)$. From now on $N$ is assumed to be
big enough so that $\mu_N(\xi)$ does not vanish on $S^{n-1}$. In
fact, we will need later on the inequality
\begin{equation}\label{eq38}
 | \partial^\alpha \mu^{-1}_N(\xi) | \leq  C,\quad |\xi|= 1,\quad
 0 \leq |\alpha| \leq 2(n+3)\,,
\end{equation}
which may be taken for granted owing to the convergence in
$C^\infty(S^{n-1})$ of the series~\eqref{eq36}. In \eqref{eq38} $C$
is a positive constant depending only on the dimension~$n$
and~$\mu$.

Notice that $T_N$ satisfies condition $(iii)$ in the Theorem (with
$T$ replaced by $T_N$), because $\mu_N(\xi) \neq 0$, $|\xi| = 1$,
and so we can apply the results of section~3. In particular,
$$
K_N (x) \chi_{\BC}(x) = T_N(b_N)(x) +  T_N(\beta_N)(x)\,,
$$
where $b_N$ and $\beta_N$ are respectively the functions $b$ and
$\beta$ defined in \eqref{eq17}. It is important to remark that
$b_N$ does not depend on $T$. As \eqref{eq11bis} shows, the
function~$b_N$ depends on $N$ only through the fundamental solution
of the operator~$\triangle^N$. The uniform estimate we need on $b_N$
is given by part~(i) of the next lemma. The polynomial estimates in
$N$ of (ii) and (iii) are also basic for the compactness argument we
are looking for.

\begin{lemma}\label{bN}
There exist a constant $C$ depending only on $n$ such that
\begin{enumerate}
\item[(i)]
$$
|\widehat{b_N}(\xi)| \leq C,\quad \xi \in \Rn\,,
$$
\item [(ii)]
$$
 \| b_N  \|_{L^\infty(B)} \leq  C \,(2N)^{2n+2}\,,
$$
and
\item [(iii)]
$$
\|\nabla b_N  \|_{L^\infty(B)} \leq  C \,(2N)^{2n+4}\,.
$$
\end{enumerate}
\end{lemma}

\begin{proof}
We first prove (i). Let $h_1,\dotsc,h_d$ be an orthonormal basis of
the subspace of $L^2(d\sigma)$ consisting of all homogeneous
harmonic polynomials of degree $2N$.
%Then $d \backsimeq (2N)^{n-2}$ \cite[p. 00]{St}\,.
As in the proof of Lemma 6 we have $h_1^2 +\dotsb+h_d^2 = d$, on
$S^{n-1}$. Set
$$
H_j(x)= \frac{1}{\gamma_{2N} \sqrt{d}}\,h_j(x), \quad x \in \Rn\,,
$$
and let $S_j$ be the higher order Riesz transform with kernel
$K_j(x) = H_j(x)/|x|^{2N+n}$. The Fourier multiplier of $S_j^2$ is
$$
\frac{1}{d}\,\frac{h_j(\xi)^2}{|\xi|^{4N}}, \quad 0 \neq \xi \in
\Rn\,,
$$
and thus
$$
\sum_{j=1}^d S_j^2 = I \,.
$$
By \eqref{eq10}, we get
$$
K_j(x)\,\chi_{\BC}(x)=  S_j(b_N)(x),\quad x \in \Rn,\quad 1 \leq j
\leq d\,,
$$
and so
\begin{equation}\label{eq39}
b_N = \sum_{j=1}^d S_j\left( K_j(x) \,\chi_{\BC}(x)\right)\,.
\end{equation}
We now appeal to a lemma of Calder\'{o}n and Zygmund (\cite{CZ}; see
\cite{LS} for a simpler proof), which can be stated as follows.

\begin{CZ}[Calder\'{o}n and Zygmund]
If $K$ is the kernel of a higher order Riesz transform, then, for
some constant $C$ depending only on $n$,
$$
|\widehat{(K(x)\,\chi_{\BC}(x))}(\xi)| \leq C \,
|\widehat{\left(P.V. \,K(x) \right)}(\xi)|, \quad \xi \in \Rn
\setminus \{0\}\,.
$$
\end{CZ}

By \eqref{eq39} and the preceding lemma, we get
\begin{equation*}
\begin{split}
 |\widehat{b_N}(\xi)| & \leq \sum_{j=1}^{d}
|\widehat{P.V.\,\,K_j(x)}(\xi)|\,|\widehat{(K_j(x)\,\chi_{\BC}(x))}(\xi)| \\
               & \leq C \,               \sum_{j=1}^{d} |\widehat{P.V.\,\,K_j(x)}(\xi)|^2 \\
& = C\,.
\end{split}
\end{equation*}

We now turn to the proof of (ii) in Lemma~\ref{bN}. In view of the
expression~\eqref{eq39} for~$b_N$, we apply Lemma~5 to the
operators~$S_j$ and the functions~$K_j(x)$, which satisfy a
Lipschitz condition on~$\BC$. We obtain
\begin{equation}\label{eq40}
\|b_N\|_\infty \leq C\, d\, \max_{1 \leq j \leq d} \|K_j\|_{CZ}\,
(\|K_j\|_{L^\infty(\BC)} + \|K_j\|_{\operatorname{Lip}(1,\BC)}) \,.
\end{equation}
As it is well known, $d \simeq (2N)^{n-2}$ \cite[p.~140]{SW}. On the
other hand
$$
\|K_j \|_{CZ} \leq  \|H_j\|_\infty+\|\nabla H_j\|_\infty\,,
$$
where the supremum norms are taken on $S^{n-1}$. Clearly
$$
\|H_j\|_\infty  = \frac{1}{\gamma_{2N}}
\|\frac{h_j}{\sqrt{d}}\|_\infty \leq \frac{1}{\gamma_{2N}} \simeq
(2N)^{n/2}\,.
$$
For the estimate of the gradient of $H_j$ we use the
inequality~\cite[p.~276]{St}
\begin{equation}\label{eq40bis}
\| \nabla H_j\|_\infty \leq C\,(2N)^{n/2+1}\,\|H_j\|_2\,,
\end{equation}
where the $L^2$ norm is taken with respect to $d\sigma$. Since the
$h_j$ are an orthonormal system,
$$
\|H_j\|_2 = \frac{1}{\sqrt{d}\,\gamma_{2N}} \simeq
\frac{(2N)^{n/2}}{(2N)^{(n-2)/2}} \simeq 2N\,.
$$
Gathering the above inequalities we get
$$
\|K_j\|_{CZ} \leq C\,(2N)^{n/2+2}\,.
$$
On the other hand, a straightforward computation yields
$$
\|K_j\|_{L^\infty(\BC)} + \|K_j\|_{\operatorname{Lip}(1,\BC)} \leq
C\, N \|H_j\|_\infty + \|\nabla H_j \|_\infty \leq C\,
(2N)^{n/2+2}\,,
$$
and therefore
$$
\| b_N  \|_{L^\infty(B)} \leq  C \,(2N)^{n-2}\,(2N)^{n/2+2} \,
(2N)^{n/2+2} = C\,(2N)^{2n+2}\,.
$$
We are only left with the proof of (iii) in Lemma~\ref{bN}.
Recalling the definition of $b$ in~\eqref{eq11bis} we see that $b_N$
has the form
$$
b_N(x) = \alpha_0 + \alpha_1 \, |x|^2 + \dotsb + \alpha_{N-1}
\,|x|^{2N-2} ,\quad |x| < 1 \,,
$$
for some real coefficients $\alpha_j$, $0 \leq j \leq N-1$. Define
the polynomial $p(t)$ of the real variable $t$ as
$$
p(t)= \alpha_0 + \alpha_1 \, t^2 +\dotsb + \alpha_{N-1} \,
t^{2N-2}\,,
$$
so that $b_N(x)= p(|x|)$, $|x| < 1$.  By part (ii) of the lemma
$$
\sup_{0 \leq t \leq 1} |p(t)| \leq C\,(2N)^{2n+2}\,,
$$
and thus, appealing to Markov's inequality \cite[p.~40]{Lo},
$$
\sup_{0 \leq t \leq 1} |p'(t)| \leq (2N-2)^2  \sup_{0 \leq t \leq 1}
|p(t)| \leq  C\,(2N)^{2n+4}\,.
$$
Now (iii) follows from the obvious identity $\frac{\partial
b_N}{\partial x_j} = p'(|x|)\, \frac{\partial |x|}{\partial x_j}$,
which gives $ |\nabla b_N (x)| \leq p'(|x|)$, $|x| < 1$.
\end{proof}

Our goal is now to show that under condition $(iii)$ of the Theorem
we can find a function $\gamma$ in $L^\infty(\Rn)$ such that
\begin{equation}\label{eq42}
 K(x)\chi_{\BC}(x)=  T(\gamma)(x),\quad x \in \Rn\,.
\end{equation}
If $T$ is a polynomial operator this was proven in the preceding
section for a $\gamma$ of the form $b+\beta$ (see \eqref{eq17}). The
approach we take up now has the advantage that when applied to $T_N$
gives a uniform bound on $\gamma_N= b_N+\beta_N$.

Since $\Omega$ has the expansion \eqref{eq33} in spherical
harmonics, we have
\begin{equation*}
\begin{split}
K(x)\chi_{\BC}(x)& = \sum_{j\geq 1}
\frac{P_{2j}(x)}{|x|^{2j+n}}\,\chi_{\BC}(x)\\*[5pt]
               & =  \sum_{j\geq 1} T_j(b_{j})(x)\,,
\end{split}
\end{equation*}
where $T_j$ is the higher order Riesz transform with kernel $
P_{2j}(x) / |x|^{2j+n}$ and $b_{j}$ is the function constructed in
section~2 (see \eqref{eq10} and \eqref{eq11bis}). The Fourier
multiplier of $T_j$ is
$$
 \gamma_{2j}\, \frac{P_{2j}(\xi)}{|\xi|^{2j}} =
\gamma_{2d}\,\frac{P(\xi)}{|\xi|^{2d}}\, \frac{\gamma_{2j}}
{\gamma_{2d}}\,\frac{Q_{2j-2d}(\xi)}{|\xi|^{2j-2d}},\quad \xi \in
\Rn \setminus \{0\}\,.
$$
Let $S_j$ be the operator whose Fourier multiplier is
\begin{equation}\label{eq42bis}
\frac{\gamma_{2j}}
{\gamma_{2d}}\,\frac{Q_{2j-2d}(\xi)}{|\xi|^{2j-2d}},\quad \xi \in
\Rn \setminus \{0\}\,,
\end{equation}
so that $T_j = R \circ S_j $. Then
\begin{equation*}
\begin{split}
  K(x)\chi_{\BC}(x)& = \sum_{j\geq d} (R \circ S_j )(b_{j})\\*[5pt]
  &= \sum_{j\geq d} T \left((U^{-1} \circ S_j )(b_{j})\right)\\*[5pt]
& =  T\left(\sum_{j\geq d} (U^{-1}\circ S_j)(b_j)\right)\,.
\end{split}
\end{equation*}
The latest identity is justified by the absolute convergence of the
series \newline $\sum_{j\geq d} (U^{-1}\circ S_j)(b_j)$ in
$L^2(\Rn)$, which follows from the estimate
\begin{equation*}
\begin{split}
\sum_{j\geq d} \|(U^{-1}\circ S_j)(b_j)\|_2 & \leq C\, \sum_{j\geq
d} \|Q_{2j-2d}\|_\infty \,\|b_j\|_{L^2(\Rn)}\\*[4pt] & \leq C\,
\sum_{j\geq d} \|Q_{2j-2d}\|_\infty \,\|b_j\|_{L^\infty(B)}\\*[5pt]
& \leq C\, \sum_{j\geq d} (2N)^{2n+2} \, \|Q_{2j-2d}\|_\infty \ <
\infty\,.
\end{split}
\end{equation*}
We claim now that the series $\sum_{j\geq d} (U^{-1}\circ S_j)(b_j)$
converges uniformly on $\Rn$ to a function $\gamma$, which will
prove \eqref {eq42} . Observe that the operator $U^{-1}\circ S_j \in
A$ is not necessarily a Calder\'{o}n-Zygmund operator because the
integral on the sphere of its multiplier does not need to vanish.
However it can be written as $ U^{-1}\circ S_j = c_j I + V_j$, where
$$
c_j = \frac{\gamma_{2j}} {\gamma_{2d}}\,\int_{S^{n-1}} \mu(\xi)^{-1}
\, Q_{2j-2d}(\xi)\,d\sigma(\xi)
$$
and $V_j$ is the Calder\'{o}n-Zygmund operator with multiplier
\begin{equation}\label{eq43}
\mu(\xi)^{-1} \frac{\gamma_{2j}}
{\gamma_{2d}}\,\frac{Q_{2j-2d}(\xi)}{|\xi|^{2j-2d}} - c_j \,.
\end{equation}
Now
$$
\sum_{j\geq d} (U^{-1}\circ S_j)(b_j) = \sum_{j\geq d} c_j\,b_j +
\sum_{j\geq d} V_j(b_j)
$$
and the first series offers no difficulties because, by Lemma 8 (ii)
and \eqref{eq34}
$$
\sum_{j\geq d} |c_j|\, \|b_j\|_{L^\infty(B)} \leq C\, \sum_{j\geq d}
(2j)^{-n/2} (2j)^{2n+2} \|Q_{2j-2d} \|_\infty < \infty\,.
$$
The second series is more difficult to treat.
 By Lemma 5 and Lemma 8 (ii) and (iii),
\begin{equation*}
\begin{split}
  \|V_j(b_j) \|_{L^\infty(\Rn)} & \leq C\, \| V_j \|_{CZ} \left(\|b_j\|_{L^\infty(B)}+ \| \nabla b_j\|_{L^\infty(B)}\right) \\*[3pt]
               & \leq C\, (2j)^{2n+4} \, \| V_j\|_{CZ} \,.
\end{split}
\end{equation*}
Estimating the Calder\'{o}n-Zygmund constant of the kernel of the
operator~$V_j$ is not an easy task, because we do not have an
explicit expression for the kernel. We do know, however, the
multiplier \eqref{eq43} of $V_j$. We need a way of estimating the
constant of the kernel in terms of the multiplier and this is what
the next lemma supplies.

\begin{lemma}
Let $V$ be a smooth homogeneous Calder\'{o}n-Zygmund operator
with\linebreak Fourier multiplier~$m$. Then for some constant $C$
depending only on $n$,
$$
\| V \|_{CZ} \leq C\, \|\triangle_{S}^{n+3} m \|_2^{1/2}\,  \| m
\|_2^{1/2}\,,
$$
where $\triangle_S$ is the spherical Laplacean and the $L^2$ norm is
taken with respect to $d\sigma$.
\end{lemma}

\begin{proof}
Let $\omega(x)/|x|^n$ be the kernel of $V$, so that $\omega$ is a
homogeneous function of degree zero, of class $C^\infty(S^{n-1})$
and with zero integral on the sphere. Consider the expansion of
$\omega$ in spherical harmonics $\omega (x)= \sum_{j \geq 1}
p_j(x)$, $|x|=1$, so that the kernel of $V$ is $\sum_{j \geq 1}
p_j(x)/|x|^{j+n}$, $x \in \Rn \setminus\{0\}$ and its Fourier
multiplier is~$m(\xi) = \sum_{j \geq 1} \gamma_j\, p_j(\xi)$,
$|\xi|=1..$ By the definition \eqref{eq29} of the constant of the
kernel of a Calder\'{o}n-Zygmund operator we have
$$
\| V \|_{CZ} \leq C\, \sum_{j \geq 1} \left(j\,\|p_j\|_\infty +
\|\nabla p_j \|_\infty \right)\,,
$$
where the supremum is taken on $S^{n-1}$. By \eqref{eq40bis} with
$H_j$ replaced by $p_j$, and Lemma~6
\begin{equation*}
\begin{split}
  \| V \|_{CZ}& \leq C\, \sum_{j \geq 1}
\left(j^{1+(n-1)/2}\,\|p_j\|_2 + j^{n/2+1}\,\| p_j \|_2\right)
\\*[5pt] &  \leq C\, \sum_{j\geq 1}  \left(j^{n/2+2}\, \| p_j \|_2
\right)\,.
\end{split}
\end{equation*}
Since $\gamma_j \simeq j^{-n/2}$, the above sum can be estimated,
using Schwarz's inequality and~\eqref{eq30bis} with $\Omega$
replaced by $m$ and $P_j$  by $\gamma_j \, p_j$, by
\begin{equation*}
\begin{split}
\sum_{j\geq 1}  j^{n+2}\, \| \gamma_j\, p_j \|_2  & \leq
C\,\left(\sum_{j\geq 1}  j^{2n+6}\, \| \gamma_j\, p_j \|_2^2
\right)^{1/2} \\*[5pt] & \leq C\,\left(\sum_{j\geq 1}
\left(j(j+n-2)\right)^{n+3}\, \|\gamma_j\, p_j\|^2_2\right)^{1/2}
\\*[5pt] & = C\,\left((-1)^{n}\,\int_{S^{n-1}} \triangle^{n+3}_S m
\, m \,d\sigma \right)^{1/2}\\*[5pt] &  \leq C\,
\|\triangle_{S}^{n+3} m \|_2^{1/2}\,  \| m \|_2^{1/2}\,.
\end{split}
\end{equation*}
\end{proof}

Since the multiplier of $V_j$ is given by \eqref{eq43} and
$\mu^{-1}$ is in $C^\infty(S^{n-1})$, Lemma~9 reduces the estimate
of $\|V_j\|_{CZ}$ to the estimate of the $L^2(d\sigma)$ norm of
$\nabla^k Q_{2j-2d}$, for $0 \leq k \leq 2(n+3)$. Let us consider
first the case $k=1$.

Since $P_{2j}= P \,\,Q_{2j-2d}$, we have
$$
\nabla P_{2j} = \nabla P \,Q_{2j-2d} +P \,   \nabla Q_{2j-2d}\,,
$$
and so, by Lemma 7 and \eqref{eq40bis} with $H_j$ replaced by
$P_{2j}$, there is a large positive integer $M = M(n,P)$ such that
\begin{equation*}
\begin{split}
 \|\nabla Q_{2j-2d}\|_\infty & \leq C\,\,j^M\, \|P\,\nabla
Q_{2j-2d}\|_\infty   \\*[5pt] & \leq C \,j^M\,\left( \|\nabla P_{2j}
\|_\infty + \|Q_{2j-2d}\|_\infty \right) \\*[5pt] & \leq C\,j^M
\,\left( C\,j^{n/2+1}\,\|P_{2j} \|_2 + C\,j^M \,\|P_{2j} \|_\infty
\right)\\*[5pt] &  \leq C\,j^M \,\|P_{2j} \|_2 \,,
\end{split}
\end{equation*}
where in the latest inequality  $M$ has been increased without
changing the notation.

By induction we get, for some large integer $M=M(n,P)$,
$$
 \|\nabla^k Q_{2j-2d}\|_\infty  \leq C\,j^M\,\|P_{2j}\|_2, \quad 0 \leq
 k \leq  2(n+3)\,.
$$
Therefore, the estimate we finally obtain for the constant of the
kernel of $V_j$ is
$$
\| V_j \|_{CZ} \leq C\,j^M \,\|P_{2j}\|_2\,,
$$
and thus
$$
  \|V_j(b_j) \|_{L^\infty(\Rn)} \leq C\,j^M \,\|P_{2j}\|_2\,,
$$
where again $M=M(n,P)$ is a positive integer. Hence the series
$\sum_{j\geq d} (U^{-1}\circ S_j)(b_j)$ converges uniformly on $\Rn$
and the proof of \eqref{eq42} is complete.

We are now ready for the discussion of the final compactness
argument that will complete the proof of the sufficient condition.
The reader is invited to review the definitions of the operators
$T_N$ (with kernel $K_N$) and $U_N$ given in this section just
before Lemma 8. We know from section 3 (see \eqref{eq17}) that
\begin{equation}\label{eq45}
K_N(x)\chi_{\BC}(x)=  T_N(b_N)(x)+ T_N(\beta_N)(x)\,.
\end{equation}
On the other hand, by the construction of the function $\gamma$ we
have just described, we also have
\begin{equation}\label{eq46}
K_N(x)\chi_{\BC}(x)=  T_N(\gamma_N)(x), \quad \gamma_N = \sum_{j\geq
d}^N (U_N^{-1}\circ S_j)(b_j)\,.
\end{equation}
Notice that \eqref{eq38} guaranties that the estimate of the
supremum norm of $\gamma$ on the whole of $\Rn$ is applicable to the
operator $T_N$, and thus we get an estimate for
$\|\gamma_N\|_{L^\infty(\Rn)}$ which is uniform in $N$. Since $T_N$
is injective, \eqref{eq45} and \eqref{eq46} imply
\begin{equation}\label{eq46bis}
b_N+\beta_N = \gamma_N
\end{equation}
and, in particular, we conclude that the functions $b_N+\beta_N$ are
uniformly bounded in $L^\infty(\Rn)$, a fact that cannot be derived
from the work done in section~3. It is worth mentioning that
numerical computations indicate that $b_N$, and thus $\beta_N$, are
not uniformly bounded.  On the other hand, section 3 tells us that
$\gamma_N$ satisfies the decay estimate \eqref{eq15} with $\beta$
replaced by $\gamma_N$, which we cannot infer from the preceding
construction of $\gamma$. The advantages of both approaches will be
combined now to get both the boundedness and decay property for
$\gamma$.

In view of \eqref{eq46} and the expressions of the multipliers  of
$U_N$ and $S_j$ (see \eqref{eq42bis}),
$$
\widehat{\gamma_N}(\xi) = \sum_{j=d}^N
\frac{1}{\mu_N(\xi)}\,\frac{\gamma_{2j}}{\gamma_{2d}}\,\frac{Q_{2j-2d}(\xi)}{|\xi|^{2j-2d}}\,
 \widehat{b_j}(\xi)\,,
$$
which yields, by Lemma 8 and \eqref{eq34} for $M=0$,
\begin{equation}\label{eq48}
\begin{split}
  \|\widehat{\gamma_N} \|_{L^\infty(\Rn)} & \leq C\,\sum_{j= d}^N \|Q_{2j-2d}
  \|_{\infty}\\*[5pt]
  & \leq C\, \sum_{j= d}^\infty \|Q_{2j-2d}\|_{\infty} \\*[5pt]
  & \leq C\,,
\end{split}
\end{equation}
where $C$ does not depend on $N$. Recall that, from \eqref{eq24bis}
in section~3, we have
$$
\beta_N =U_N^{-1}(\beta_{1,N})\,,
$$
with $\beta_{1,N}$ a bounded function supported on $B$  satisfying
$\int \beta_{1,N}(x) \,dx =0$. Since
$$
\widehat{\beta_{1,N}} = \mu_N \, \widehat{\beta_{N}} = \mu_N
\,(\widehat{\gamma_{N}} - \widehat{b_{N}})\,,
$$
we have, again by Lemma 8,
$$
 \|\widehat{\beta_{1,N}} \|_{L^\infty(\Rn)} \leq C\,.
$$
Therefore, passing to a subsequence, we may assume that, as $N$ goes
to $\infty$,
$$
\widehat{b_{N}}\longrightarrow a_0 \quad\quad\quad
{\text{and}}\quad\quad\quad \widehat{\beta_{1,N}} \longrightarrow
a_1\,,
$$
weak $\star$ in $L^\infty(\Rn)$. Hence
$$
b_{N}\longrightarrow \Phi_0 = \mathcal{F}^{-1}{a_0} \quad\quad\quad
{\text{and}}\quad\quad\quad \beta_{1,N} \longrightarrow \Phi_1 =
 \mathcal{F}^{-1}{a_1}\,,
$$
in the weak $\star$ topology of tempered distributions,
$\mathcal{F}^{-1}$ being the inverse Fourier transform. In
particular, $\Phi_0$ and $ \Phi_1$ are distributions supported on
$\overline B$  and
\begin{equation}\label{eq49}
\langle\Phi_1,1\rangle = \lim_{N \rightarrow \infty} \int
\beta_{1,N}(x)\,dx = 0\,.
\end{equation}

We would like now to understand the convergence properties of the
sequence of the $\beta_N$'s . Since
$$
\widehat{\beta_{N}}(\xi) =  \mu_N^{-1}(\xi) \,
\widehat{\beta_{1,N}}(\xi)\,,
$$
and we have pointwise bounded convergence of $\mu_N^{-1}(\xi)$
towards $\mu^{-1}(\xi)$ on $\Rn \setminus\{0\}$, we get that
$\widehat{\beta_{N}} \rightarrow \mu^{-1}\,a_1$, in the weak $\star$
topology of $L^\infty(\Rn)$. Thus $\beta_{N} \rightarrow
U^{-1}(\Phi_1)$ in the weak $\star$ topology of tempered
distributions. Letting $N \rightarrow \infty$ in \eqref{eq46bis} we
obtain
$$
\Phi_0 + U^{-1}(\Phi_1) = \gamma\,.
$$
We come now to the last key point of the proof, namely, that one has
decay estimate
\begin{equation}\label{eq50}
|\gamma(x)| \leq \frac{C}{|x|^{n+1}},\quad |x|\geq 2\,.
\end{equation}
Since $\Phi_0$ and $\Phi_1$ are supported on $\overline B$ and
$U^{-1}(\Phi_1) = \lambda\,\Phi_1 + V(\Phi_1 )$, where $\lambda$ is
a real number and $V$ a smooth homogeneous Calder\'{o}n-Zygmund
operator, it is enough to show that $V(\Phi_1 )$ has the appropriate
behavior off the ball $B(0,2)$. Let $L$ be the kernel of $V$.
Regularizing $\Phi_1$ one checks that, for a fixed $x$ satisfying
$|x|\geq 2$,
\begin{equation}\label{eq51}
\begin{split}
V(\Phi_1 )(x) & = \langle \Phi_1, L(x-y)\rangle \\*[3pt] & =
\langle\Phi_1, L(x-y)-L(x)\rangle \,,
\end{split}
\end{equation}
where the latest identity follows from \eqref{eq49}. Since $\Phi_1 $
is a distribution supported on~$\overline B$ there exists a positive
integer $\nu$ and a constant $C$ such that
\begin{equation}\label{eq52}
|\langle\Phi_1,\varphi\rangle | \leq  C\sup_{|\alpha|\leq \nu}
\sup_{|y|\leq 3/2} |\partial^\alpha \varphi(y)|\,,
\end{equation}
for each infinitely differentiable function $\varphi$ on $\Rn$. The
kernel $L$ satisfies
$$
 |\frac{\partial^\alpha}{{\partial y}^\alpha} \left(
L(x-y)-L(x)\right)| \leq \frac{C_\alpha}{|x|^{n+1+|\alpha|}},\quad
|y| \leq 3/2\,,
$$
and hence by \eqref{eq51} and \eqref{eq52}
$$
|V(\Phi_1)(x)| \leq \frac{C}{|x|^{n+1}},\quad |x|\geq 2\,,
$$
which proves \eqref{eq50} and then completes the proof of the
sufficient condition in the general case.

\section{Proof of the necessary condition: the polynomial case }

We assume in this section that $T$ is a polynomial operator with
kernel
$$
K(x)=\frac{\Omega(x)}{|x|^n}=
\frac{P_2(x)}{|x|^{2+n}}+\frac{P_4(x)}{|x|^{4+n}}+\dotsb+\frac{P_{2N}(x)}{|x|^{2N+n}},\quad
x \neq 0\,,
$$
where $P_{2j}$ is a homogeneous harmonic polynomials of degree $2j$.
Let $Q$ be the  homogeneous polynomial of degree $2N$ defined by
$$
 Q(x)=  \gamma_2 \, P_2(x)|x|^{2N-2}+ \dotsb+ \gamma_{2j}\,
P_{2j}(x)|x|^{2N-2j}+\dotsb+ \gamma_{2N}\,P_{2N}(x) \,.
$$
Then
$$
\widehat{P.V.K}(\xi) = \frac{Q(\xi)}{|\xi|^{2N}},\quad \xi \neq 0\,.
$$
 Our assumption is now the $L^2$ estimate between $T^{\star}$ and $T$ (see $(ii)$ in the statement of
the Theorem). Since the truncated operator $T^1$ at level $1$ is
obviously dominated by $T^{\star}$, we have
$$
\int (T^1f)^2 (x)\,dx \leq \int (T^{\star}f)^2 (x)\,dx \leq C\,\int
(Tf)^2 (x)\,dx \,.
$$
The kernel of $T^1$ is (see \eqref{eq14})
\begin{equation}\label{eq53}
K(x)\,\chi_{\BC}(x)= T(b)(x) + S(x)\,\chi_{B}(x)\,,
\end{equation}
where $b$ is given in equation \eqref{eq11bis} and
$$
-S(x)= Q(\partial)(A_0+A_1\,|x|^2+\dotsb+
A_{2N-1}\,|x|^{4N-2})(x),\quad x \in \Rn \,.
$$
The reader may consult the beginning of section 3 to review the
context of the definition of $S$. In view of \eqref{eq53} we have,
for each $f  \in L^2(\Rn)$,
\begin{equation*}
\begin{split}
  \|S\,\chi_{B} \star f \|_2  & \leq C\, \|T^1 f \|_2 + \|b \star Tf \|_2 \\*[3pt]
& \leq C\, (\|Tf \|_2 +\|\widehat{b} \|_\infty \|Tf \|_2) \\*[3pt] &
= C\,\|Tf\|_2 \,.
\end{split}
\end{equation*}
By Plancherel, the above $L^2$ inequality translates into a
pointwise inequality between the Fourier multipliers, namely,
\begin{equation}\label{eq54}
|\widehat{S\,\chi_B}(\xi)| \leq C\, |\widehat{P.V. K}(\xi)| = C\,
\frac{|Q(\xi)|}{|\xi|^{2N}}\,.
\end{equation}
Our next goal is to show that \eqref{eq54} provides interesting
relations between the zero sets of $Q$ and the $P_{2j}$. For each
function $f$ on $\Rn$ set $Z(f)= \{x \in \Rn : f(x)=0 \}$.

\begin{lemanom}[Zero Sets Lemma]
$$
Z(Q) \subset Z(P_{2j}), \quad 1 \leq j \leq N \,.
$$
\end{lemanom}

\begin{proof}
We know that $S$ has an expression of the form (see \eqref{eq21})
\begin{equation}\label{eq54bis}
S(x)= \sum_{l=N+1}^{2N-1} \sum_{j=1}^{l-N}
c_{lj}\,P_{2j}(x)\,|x|^{2(l-N-j)}\,.
\end{equation}
Since $\widehat{\chi_B} = G_m$, $m=n/2$, Lemma 3 yields
\begin{equation}\label{eq54tris}
\begin{split}
\widehat{S\,\chi_{B}}(\xi) & = S(\imath\, \partial)\,
\widehat{\chi_B}(\xi)\\*[5pt] & =  \sum_{l=N+1}^{2N-1}
\sum_{j=1}^{l-N}
c_{lj}\,(-1)^{l-N}\,\,P_{2j}(\partial)\,\,\triangle^{l-N-j}\,G_{m}(\xi)\\*[5pt]
& = \sum_{l=N+1}^{2N-1} \sum_{j=1}^{l-N} \sum_{k=0}^{l-N-j}
c_{\,l,\,j,\,k}\,P_{2j}(\xi)\,|\xi|^{2(l-N-j-k)}
G_{m+2l-2N-k}(\xi)\,.
\end{split}
\end{equation}
The function $G_p(\xi)$ is, for each $p \geq 0$, a radial function
which is the restriction to the real positive axis of an entire
function \cite[A-8]{Gr}. Set $\xi= r\, \xi_0$, $|\xi_0|=1$,
$r\geq0$. Then
\begin{equation}\label{eq55}
\widehat{S \chi_B}(r\xi_0)= \sum_{p=1}^{\infty}
a_{2p}(\xi_0)\,r^{2p} \,,
\end{equation}
and the power series has infinite radius of convergence for each
$\xi_0$. Assume now that $Q(\xi_0) = 0$. Then, by \eqref{eq54},
$\widehat{S \chi_B}(r\xi_0)= 0$ for each $r \geq 0$, and hence
$a_{2p}(\xi_0) = 0$, for each $p \geq 1$. For $p=1$ one has
$a_{2}(\xi_0) = P_2(\xi_0)\,C_2$, where
$$
C_2 = \sum_{l=N+1}^{2N-1} c_{\,l,\,1,\,l-N-1}\,G_{m+l-N+1}(0)\,.
$$
It will be shown later that $C_2 \neq 0$, and then we get
$P_2(\xi_0) =0 $. Let us make the inductive hypothesis that
$P_2(\xi_0)= \dotsb=P_{2(j-1)}(\xi_0)=0$. Then we obtain, if $j \leq
N-1$, $a_{2j}(\xi_0) = P_{2j}(\xi_0)\,C_{2j}$, where
\begin{equation}\label{eq55bis}
C_{2j} = \sum_{l=N+j}^{2N-1} c_{\,l,\,j,\,l-N-j}\,G_{m+l-N+j}(0)\,.
\end{equation}

Since we will show that $C_{2j} \neq 0$,  $P_{2j}(\xi_0) =0$, $1
\leq j \leq N-1$. We have
$$
0= Q(\xi_0)=  \sum_{j=1}^N \gamma_{2j} \, P_{2j}(\xi_0)\,,
$$
and so we also get $P_{2N}(\xi_0)=0$. Therefore the zero sets Lemma
is completely proved provided we have at our disposition the
following.

\begin{lemma}\label{L10}
$$
C_{2j} = \frac{- \pi^\n}{V_{n}\, 2^\n \, \Gamma(\n+1)} \;
\frac{(-1)^j }{j\, 4^{j} \, \Gamma (2j+\n)},\quad 1 \le j \le N-1
\,.
$$
\end{lemma}
\noindent The proof of Lemma 10 is lengthy and rather complicated
from the computational point of view, and so we postpone it to
section 7.
\end{proof}

Notice that, although the constants $C_{2j}$ are non-zero, they
become rapidly small  as the index $j$ increases.

The reason why Lemma 10 is involved is that one has to trace back
the exact values of the constants $C_{2j}$ from the very beginning
of our proof of \eqref{eq53}. This forces us to take into account
the exact values of various other constants. For instance, those
which appear in the expression of the fundamental solution of
$\triangle^N$ and the constants $A_0, A_1,\dotsc, A_{2N-1}$ in
formula \eqref{eq11}. Finally,  we need to prove some new identities
involving a triple sum of combinatorial numbers, in the spirit of
those that can be found in the book of R.\ Graham D. Knuth and O.\
Patashnik \cite{GKP}.

The following is  elementary folklore, but is proved here for the
reader's sake.

\begin{lemanom}[Dimension Lemma]
If $f$ is a real valued continuous function on $\Rn$\linebreak which
changes sign, then $H^{n-1}(Z(f)) > 0$, $H^{n-1}$ being Hausdorff
measure in di\-mension~~$n-1$. In particular, the Hausdorff
dimension of $Z(f)$ is at least $n-1$.
\end{lemanom}

\begin{proof}
Assume, without loss of generality, that $f(0)> 0$ and $f(p)< 0$,
where $p=(0,\dotsc,0,1)$. For $\epsilon>0$ small enough we have
$f(x) > 0$ if $|x|< \epsilon$ and $f(x)< 0$ if $|x -p|< \epsilon$.
Set $B= \{x \in \Rn : x_n=0\quad {\text{and}}\quad |x|<\epsilon \}$.
Bolzano's theorem tells us that, for each $x \in B$, $f$ vanishes at
some point of the segment $(x_1,\dotsc,x_{n-1},t)$, $0 \leq t \leq
1$. Hence the orthogonal projection of the set $Z(f)$ onto the
hyperplane~$\{x:x_n =0\}$  contains $B$ and so $H^{n-1}(Z(f)) \geq
H^{n-1}(B)
>0$.
\end{proof}

We turn now our attention to an algebraic lemma which plays a key
role in obtaining the necessary condition we are looking for.

\begin{lemanom}[Division Lemma]\label{divisio}
Let $F$ and $G$ be polynomials in $\mathbb{R}[x_1,\dotsc,x_n]$.
Assume that $G$ is irreducible and that $H^{n-1}(Z(F)\cap Z(G)) >
0$. Then there exists a polynomial $H$ in
$\mathbb{R}[x_1,\dotsc,x_n]$ such that $F=G\,H$.
\end{lemanom}

\begin{proof}
Denote by $V(P)$ the complex hyper-surface  $ \{z \in \mathbb{C}^n :
P(z) = 0 \}$ of a polynomial $P$. By hypothesis $V(F)\cap V(G)$ is
not empty. If $V(G)$ is not contained in $V(F)$ then the complex
dimension of $V(G)\cap V(F)$ is not greater than $n-2$ \cite[3.2
p.~131]{K}. Since the real dimension of a variety is less than or
equal to the complex dimension, we conclude that $Z(G)\cap Z(F)$ has
real dimension not greater than $n-2$, which contradicts the fact
that it has positive $n-1$-dimensional Hausdorff measure. Thus $V(G)
\subset V(F)$, and therefore $F=G\,H$ for some polynomial $H$ in
$\mathbb{C}[x_1,\dotsc,x_n]$. Since $F$ and $G$ have real
coefficients, the same happens to $H$.
\end{proof}

We proceed to the proof of the necessary condition.

Let $j_0$ be the first positive index such that $P_{2j_0}$ does not
vanish identically. We want to show that $P_{2j_0}$~divides $P_{2j}$
for $j_0 \leq j \leq N$.

Since $\mathbb{R}[x_1,\dotsc,x_n]$ is a unique factorization domain
we can express $P_{2j_0}$  as a product of irreducible factors, say
$R_k$, $1 \leq k \leq M$, which are also homogeneous. Clearly
$Z(P_{2j_0}) = \cup_k Z(R_k)$ and so
$$
Z(Q) =  \bigcup_k (Z(Q)\cap Z(R_k))\,.
$$
Since the integral of $Q$ on the sphere is $0$, $Q$ changes sign and
thus by the Dimension Lemma there is at least a $k$ such that
$H^{n-1}(Z(Q)\cap Z(R_k)) > 0$.  Change notation if necessary so
that $k=1$. Then $R_1$ divides $Q$, by the Division Lemma. We may
also apply the Division Lemma to $R_1$ and $P_{2j}$ for each $j_0
\leq j \leq N$, because  $Z(Q)\cap Z(R_1) \subset Z(P_{2j})\cap
Z(R_1) $ by the Zero Sets Lemma. Hence $R_1$ also divides~$P_{2j}$,
for $j_0 \leq j \leq N$. Set
$$
Q= R_1\,Q_1 \qquad  {\text{and}} \qquad P_{2j} = R_1\,
P_{2j,1},\quad j_0 \leq j \leq N\,,
$$
for certain homogeneous polynomials $Q_1$ and $P_{2j,1}$.

If $M=1$ we are done. Otherwise our intention is to repeat as many
times as we can the above division process. With this in mind we use
\eqref{eq55} to rewrite inequality~\eqref{eq54} in the form
\begin{equation}\label{eq56}
\left| \sum_{p=1}^{\infty} a_{2p}(\xi_0)\,r^{2p}\right| \leq C\,
|Q(\xi_0)| ,\quad 0< r \,.
\end{equation}
The definition of the coefficients $a_{2p}$ and \eqref{eq54tris}
show that there exist real numbers~$\mu_j(p)$ such that
$$
a_{2p}(\xi_0) = \sum_{j=j_0}^{N-1} \mu_j(p)\,P_{2j}(\xi_0)
$$
and so
\begin{equation}\label{eq57}
\begin{split}
 a_{2p}(\xi_0) & =   R_1(\xi_0) \,\sum_{j=j_0}^{N-1}
\mu_j(p)\,P_{2j,1}(\xi_0) \\
& =  R_1(\xi_0)\, a_{2p,1}(\xi_0) \,,
\end{split}
\end{equation}
where the last identity provides the definition of the numbers $
a_{2p,1}(\xi_0)$. We can simplify the common factor $R_1(\xi_0)$ in
\eqref{eq56} and get
\begin{equation}\label{eq58}
\left| \sum_{p=1}^{\infty} a_{2p,1}(\xi_0)\,r^{2p}\right| \leq C\,
|Q_1(\xi_0)| ,\quad 0< r \,.
\end{equation}
Equipped with \eqref{eq58}, we are ready to begin the second step in
the division process. If $Q_1(\xi_0)=0 $, then $a_{2p,1}(\xi_0) =
0$, for each $p \geq 1$. Hence, as in the proof of the Zero Sets
Lemma,
$$
Z(Q_1) \subset Z(P_{2j,1}), \quad j_0 \leq j \leq N \,.
$$
To apply the Division Lemma we need to ascertain that the zero set
of $Q_1$ is big enough and for that it suffices to show, by the
Dimension Lemma, that $Q_1$ changes sign. As we are assuming that
$M$ is greater than $1$, the degree of $R_1$ is less than $2j_0$.
Considering the expansions of $R_1$ and $Q$ in spherical harmonics,
we see that they are orthogonal in $L^2(d\sigma)$ \cite[p.~69]{St}.
Hence
$$
0= \int R_1(\xi)\,Q(\xi) \,d\sigma(\xi) = \int R_1^2(\xi)\,Q_1(\xi)
\,d\sigma(\xi)\,,
$$
which tells us that $Q_1$ changes sign.

Since $P_{2j_0,1} = \prod_{k=2}^M R_k $, we conclude that one of the
$R_k$, say $R_2$, divides the $P_{2j,1}$, $j_0 \leq j \leq N$. An
inductive argument gives that $P_{2j_0}$ divides the $P_{2j}$, $j_0
\leq j \leq N$. At the $k$-th step one should observe that $Q=
\prod_{l=1}^k R_l \,Q_k$ and
$$
0= \int \prod_{l=1}^{k}R_l(\xi)\,Q(\xi) \,d\sigma(\xi) = \int
\prod_{l=1}^{k}R_l^2(\xi)\,Q_k(\xi)  \,d\sigma(\xi)\,,
$$
so that $Q_k$ changes sign.  It is also important to remark that we
have
$$
a_{2p,k}(\xi_0) = \sum_{j=j_0}^{N-1} \mu_j(p)\,P_{2j,k}(\xi_0),\quad
1 \leq k \leq M\,.
$$
Thus at the $M$-th step we get for $p=j_0$
\begin{equation}\label{eq58bis}
a_{2j_0,M}(\xi_0) = C_{2j_0} \neq 0\,.
\end{equation}
We have $P_{2j}= P_{2j_0}\,Q_{2j-2j_0}$ for some homogeneous
polynomials $Q_{2j-2j_0}$ of degree~$2j-2j_0$ and so
\begin{equation*}%\label{eq57}
\begin{split}
 Q(\xi) & =    \sum_{j=j_0}^N   \gamma_{2j}\,
P_{2j}(\xi)|\xi|^{2N-2j} \\
& =  P_{2j_0}(\xi)\, \sum_{j=j_0}^N   \gamma_{2j}\,
Q_{2j-2j_0}(\xi)|\xi|^{2N-2j}  \,.
\end{split}
\end{equation*}
By \eqref{eq56} and the definition of the coefficients
$a_{2p,M}(\xi_0)$, for $|\xi_0|=1$ and $0<r$, we get
$$
\left| \sum_{p=j_0}^{\infty} a_{2p,M}(\xi_0)\,r^{2p}\right| \leq C
\left| \sum_{j=j_0}^N   \gamma_{2j}\,
Q_{2j-2j_0}(\xi_0)\right|,\quad 0< r \,.
$$
Taking into account \eqref{eq58bis} we conclude that
$$
\sum_{j=j_0}^N \gamma_{2j}\, Q_{2j-2j_0}(\xi_0) \neq 0,\quad
|\xi_0|= 1 \,,
$$
which completes the proof of the necessary condition in the
polynomial case.

\section{Proof of the necessary condition: the general case }

In this section the kernel of our operator has the general form
$K(x)=\Omega(x)/|x|^n$ with $\Omega$ a homogeneous function of
degree $0$, with vanishing integral on the sphere and of class
$C^\infty(S^{n-1})$. Then $\Omega(x) = \sum_{j\geq 1}^\infty
P_{2j}(x)/|x|^{2j}$ with $P_{2j}$ a homogeneous harmonic polynomial
of degree $2j$. The strategy consists in passing to the polynomial
case by looking at a partial sum of the series above. Set, for each
$N \geq 1$, $K_N(x)=\Omega_N(x)/|x|^n$, where $\Omega_N(x) =
\sum_{j= 1}^N P_{2j}(x)/|x|^{2j}$, and let $T_N$ be the operator
with kernel $K_N$. The difficulty is now that there is no obvious
way of obtaining the inequality
\begin{equation}\label{eq58tris}
\| T_N^{\star} f  \|_2 \leq  C \| T_N f \|_2,\quad f \in L^2(\Rn)
\,,
\end{equation}
from our hypothesis, namely,
$$
 \| T^{\star} f  \|_2 \leq  C \| T f \|_2,\quad f \in L^2(\Rn) \,.
$$
Instead we try to get \eqref{eq58tris} with $\| T_N f \|_2$ replaced
by $\| Tf \|_2$ in the right hand side plus an additional term which
becomes small as $N$ tends to $\infty$. We start as follows.
\begin{equation*}%\label{eq57}
\begin{split}
  \| T^1_Nf\|_2 & \leq  \| T^1f\|_2 + \| T^1f -T^1_N f\|_2  \\*[5pt]
&  \leq C\, \|Tf\|_2 +\| \sum_{j>N} \frac{P_{2j}(x)}{|x|^{2j+n}}\,
\chi_{\BC}* f\|_2  \,.
\end{split}
\end{equation*}
By \eqref{eq10} there exists a bounded function $b_j$ supported on
$B$ such that
$$
\frac{P_{2j}(x)}{|x|^{2j+n}}\, \chi_{\BC}(x) = P.V.
\frac{P_{2j}(x)}{|x|^{2j+n}}* b_j\,.
$$
By Lemma 8 (i) $\|\widehat{b_j}\|_{L^\infty(\Rn)}$ is bounded
uniformly in $j$ and then an application of Plancherel yields
\begin{equation*}%\label{eq57}
\begin{split}
\left  \| \sum_{j>N} \frac{P_{2j}(x)}{|x|^{2j+n}}\, \chi_{\BC}*
f\right\|_2 & = \left \| \sum_{j>N} P.V.
\frac{P_{2j}(x)}{|x|^{2j+n}}* b_j * f\right\|_2  \\*[5pt] &  \leq
C\,  \left( \sum_{j>N} \|P_{2j}\|_\infty \right) \,\|f\|_2\,,
\end{split}
\end{equation*}
where the supremum norm is taken on the sphere. By \eqref{eq53}
applied to $T_N$
$$
T^1_N f=K_N\, \chi_{\BC}* f= K_N * b_N * f + S_N\, \chi_B * f \,.
$$
Hence, for each $f$ in $L^2(\Rn)$\,,
\begin{equation*}
\begin{split}
 \| S_N\chi_B *f\|_2 & \le  \|T^1_N f\|_2 + \|K_N* f * b_N \|_2 \\*[5pt]
      &  \le C\| Tf\|_2 + C\, \left( \sum_{j>N} \|P_{2j}\|_\infty \right) \,\|f\|_2 + \| K_ N * f * b_N\|_2 \\*[5pt]
     &    \le C\| Tf\|_2+  \|Tf * b_N
     \|_2 + C \,\left( \sum_{j>N} \|P_{2j}\|_\infty \right) \,\|f\|_2 \\*[5pt]
     & \le C\| Tf\|_2 + C \,\left( \sum_{j>N} \|P_{2j}\|_\infty \right) \,\|f\|_2     \,,
\end{split}
\end{equation*}
where in the latest inequality Lemma 8 (i) was used. The above $L^2$
inequality translates, via Plancherel, into the pointwise estimate
 \begin{equation}\label{eq59}
| \widehat{S_N \chi_B} (\xi)| \le C| \widehat{P.V. K} (\xi) |+ C\,
\left( \sum_{j>N} \|P_{2j}\|_\infty \right), \quad \xi \neq 0 \,.
\end{equation}
The idea is now to take limits, as N goes to $\infty$, in the
preceding inequality. The remainder of the convergent series will
disappear and we will get a useful analog of~\eqref{eq54}. The first
task is to clarify how the left hand side converges.

Set $\xi= r\,\xi_0$, with $|\xi_0| = 1$ and $ r>0$. Rewrite
\eqref{eq55} with $S$ replaced by $S_N$ and $a_{2p}$ by $a_{2p}^N$:
$$
\widehat{S_N \chi_B}(r\xi_0)= \sum_{p=1}^{\infty} a_{2p}^N
(\xi_0)\,r^{2p} \,.
$$
It is a remarkable key fact that for a fixed $p$  the sequence of
the $a_{2p}^N$ stabilizes for $N$~large. This fact depends on a
laborious computation of various constants and will be proved in
section 7 in the following form.

\begin{lemma}\label{L11}
If $p+1 \leq N $, then $a_{2p}^N = a_{2p}^{p+1}$.
\end{lemma}

If $p\geq1$ and $p+1 \leq N $ we set $a_{2p}= a_{2p}^N$. We need an
estimate for the $a_{2p}^N$, which will be proved as well in section
7.

\begin{lemma}\label{L12}
We have, for a constant $C$ depending only on $n$,
\begin{equation}\label{eq59bis}
|a_{2p}| \leq \frac{C}{(p-1)!\,4^p} \,\sum_{j=1}^p
\|P_{2j}\|_\infty, \quad 1 \le p \le N-1\,,
\end{equation}
and
\begin{equation}\label{eq59tris}
|a_{2p}^N | \le \frac{C}{4^p} \, \binom{\frac{n}{2}+N-1}
{N-1}\,\sum_{j=1}^{N-1} \|P_{2j}\|_\infty , \quad 1 < N \le p\,.
\end{equation}
\end{lemma}

Let us prove that for each $\xi_0$ in the sphere the sequence
$S_N\,\chi_B (r\,\xi_0)$ converges uniformly on $0 \le r \le 1$. For
$1 \le N \le M$
\begin{equation*}
\begin{split}
 \left| \widehat{S_{N}\, \chi_{B}}(r \xi_{0}) - \widehat{S_{M}\, \chi_{B}}(r\xi_{0}) \right|
 & \le \sum_{p \ge N} | a_{2p}^{N}| r^{2p} + \sum_{p= N}^{M-1} | a_{2p}| r^{2p} + \sum_{p\ge M} | a_{2p}^{M}| r^{2p} \\*[5pt]
 & \le C \left( \frac{1}{4^N} \, \binom{\frac{n}{2}+N-1}
{N-1} +    \sum_{p\ge N} \frac{1}{(p-1)!\,4^p}
\right)\sum_{j=1}^{\infty} \|P_{2j}\|_\infty \,,
\end{split}
\end{equation*}
which clearly tends to $0$ as $N$ goes to $\infty$. Letting $N$ go
to $\infty$ in \eqref{eq59} we get
\begin{equation}\label{eq60}
\left |\sum_{p=1}^{\infty} a_{2p} (\xi_0)\,r^{2p}\right| \le C\,
|\widehat{P.V. K} (\xi_0)|, \quad 0 \leq r \leq 1, \quad |\xi_0| =
1\,.
\end{equation}
At this point we may repeat almost verbatim the proof we presented
in the previous section, because the coefficients $ a_{2p}^{N}$
stabilize. This allows us to argue as in the polynomial case. The
only difference lays in the fact that now we are dealing with
infinite sums. However, no convergence problems will really arise
because of \eqref{eq32}.

\section{Proof of the combinatorial Lemmata}

This section will be devoted to prove Lemmas 10, 11 and 12 stated
and used in the preceding sections.

For the proof of Lemma 10 (see section 5)  we need to carefully
trace back the path that led us to the constants $C_{2j}$. To begin
with we need a formula for the coefficients $A_l$ in \eqref{eq11}
and for that it is essential to have the expression for the
fundamental solution $E_N= E_N^n$ of $\triangle^N$ in $\Rn$. One has
\cite{ACL}
$$
E_N(x)=\frac{1}{|x|^{n-2N}}(\alpha(n,N)+\beta(n,N)\log|x|^2)\,,
$$
where $\alpha$ and $\beta$ are constants that depend on $n$ and $N$.
To write in close form $\alpha$ and~$\beta$ we consider different
cases. Let $\omega_n$ be the surface measure of $S^{n-1}$.

\vspace*{7pt}

\noindent {\sl Case 1:} $n$ is odd. Then
$$
\alpha(n,N)=
\frac{\Gamma(2-\frac{n}{2})}{4^{N-1}\,(N-1)!\,\Gamma(N+1-\frac{n}{2})\,(2-n)\,\omega_n}
$$
and
$$
\beta(n, N)=0\,.
$$

\vspace*{7pt}

\noindent {\sl Case 2:} $n$ is even, $n \ne 2$ and $N\le
\frac{n}{2}-1$. Then
$$
\alpha(n,N)=\frac{(-1)^{N-1}\,(\frac{n}{2}-N-1)!}{4^{N-1}(N-1)!\,(\frac{n}{2}-2)!\,(2-n)\,\omega_n}
$$
and
$$
\beta(n, N)=0\,.
$$

\vspace*{7pt}

\noindent {\sl Case 3:}   $n$ is even, $n \ne 2$ and $N\ge
\frac{n}{2}$. Then
$$
\beta(n,N)=\frac{1}{(-1)^{\frac{n}{2}+1}(N-\frac{n}{2})!\,4^{N-1}\,(N-1)!\,(\frac{n}{2}-2)!\,(2-n)\,\omega_n}
$$
and
$$
\alpha(n,N)=2\,\beta(n,N)\,S_{N-\frac{n}{2}}\,,
$$
where $S_0=0$ and
$$
S_L= \sum_{k=1}^{L} \frac{1}{2k}+ \sum_{k=\n}^{L+\n-1}\frac{1}{2k},
\quad 1 \le L\,.
$$

\vspace*{7pt}

\noindent {\sl Case 4:} $n=2.$
$$
\beta(2,N)= \frac{1}{2}\,\frac{1}{ 4^{N-1}(N-1)!^2\,\omega_2}
$$
and
$$
\alpha(2,N)=2\,\beta(2,N)\,S_{N-1}\,.
$$

Recall that the constants  $A_0, A_1,\dotsc, A_{2N-1}$ are chosen so
that the function (see \eqref{eq11})
\begin{equation}
\varphi(x)= E(x)\,\chi_{\BC}(x) + (A_0+A_1\,|x|^2 +\dotsb+
A_{2N-1}\,|x|^{4N-2})\,\chi_B(x)\,,
\end{equation}
and all its partial derivatives of order not greater than $2N-1$
extend continuously up to $\partial B$.

\begin{lemma}\label{L13}
For $L=N+1, \dotsc, 2N-1$ we have
$$
A_{L}= \frac{(-1)^{N+L}\displaystyle\binom{N+\n-1}{N-1}}{V_n\, 4^N
\, (L+\n-N)\,(2N-L-1)!\,L!}\,,
$$
where $V_n$ is the volume of the unit ball.
\end{lemma}

\begin{proof}
Set $t=|x|^2$, so that
\begin{equation} \label{eq69}
E_N^n(x)\equiv E(t)=t^{N-\n}(\alpha+\beta\log(t))\,.
\end{equation}

Let $P(t)$ be the polynomial $\sum_{L=0}^{2N-1}A_Lt^L.$ By
Corollary~2 in section~2 we need that
$$
P^{k)}(1)=E^{k)}(1) ,  \quad  0 \le k \le 2N-1\,.
$$

By Taylor's expansion we have that
$P(t)=\sum_{i=0}^{2N-1}\frac{E^{i)}(1)}{i!}(t-1)^{i},$ and hence, by
the binomial formula applied to $(t-1)^i$,
$$
A_L=\sum_{i=L}^{2N-1}\frac{E^{i)}(1)}{i!}(-1)^{i-L}\binom{i}{L},
\quad  0 \le L \le 2N-1\,.
$$
Now we want to compute $E^{i)}(1)$.  Clearly
$$
\left( \frac{d}{dt}\right)^i (t^{N-\n})=\left(N-\n \right) \cdots
\left(N-\n-i+1\right)t^{N-\n-i}.
$$
Notice that the logarithmic term in \eqref{eq69} only appears when
the dimension  $n$ is even. In this case, for each $i\geq N+1$
$$
\left( \frac{d}{dt}\right)^i(t^{N-\n}\log t )=
\left(N-\n\right)!(-1)^{i-N+\n-1}\left(i-N+\n-1\right)!\;t^{-i+N-\n}\,.
$$
Hence,  for $i\geq N+1$, we obtain
\begin{multline*}
E^{i)}(1)=\alpha(n,N)\left(N-\n\right)\cdots\left(n-\n-i+1\right)\\*[5pt]
+\beta(n,N)\left(N-\n\right)!(-1)^{i-N+\n-1}\left(i-N+\n-1\right)!
\end{multline*}
Consequently,
\begin{equation} \label{eq70}
\begin{split}
A_L & =(-1)^L\alpha(n,N)\sum_{i=L}^{2N-1}\left(N-\n\right)\cdots
\left(N-\n-i+1\right)\frac{(-1)^i}{i!}\binom{i}{L}  \\*[5pt] &\quad+
(-1)^{L-N+\n-1}\beta(n,N)\left(N-\n\right)!
\sum_{i=L}^{2N-1}\left(i-N+\n-1\right)!\frac{\binom{i}{L}}{i!}\,.
\end{split}
\end{equation}
Let's remark that for the cases $n=2$ or $n$ even and $N\geq\n$ the
first term in \eqref{eq70} is zero, while for the cases $n$ odd or
$n$ even and $N\leq \n-1$ the second term is zero because
$\beta(n,N) = 0.$ This explains why we compute below the two terms
separately.

For the first term we show that
\begin{equation}\label{eq71}
\sum_{i=L}^{2N-1}\!\left(N\!-\!\n\right)\cdots
\left(N-\n-i+1\right)\frac{(-1)^i}{i!} \binom{i} {L}\! =\! (-1)^L
\!\binom{N-\n}{L}\!\binom{\n+N-1}{2N-1-L}\,.\!\!
\end{equation}
Indeed, the left hand side of \eqref{eq71} is, setting $k=i-L$,
%$$
%\sum_{i=L}^{2N-1}(N-\n)\cdots (N-\n-i+1)(-1)^i\frac{1}{L!(i-L)!}
%$$
\begin{equation*}
\begin{split}
%& \sum_{i=L}^{2N-1}(N-\n)\cdots (N-\n-i+1)(-1)^i\frac{1}{L!(i-L)!} \\
\frac{1}{L!}\sum_{k=0}^{2N-1-L}&\left(N-\n\right)\cdots\left(N-\n-L-k+1\right)\frac{(-1)^{L+k}}{k!}\\*[7pt]
& = (-1)^L\binom{N-\n}{L}
\sum_{k=0}^{2N-1-L}\binom{\n+L-N+k-1}{k}\\*[7pt] & = (-1)^L
\binom{N-\n}{L}\binom{N+\n-1}{2N-1-L}\,,
\end{split}
\end{equation*}
where the last identity comes from (\cite[(5.9), p.~159]{GKP}).

To compute the second term we first show that
\begin{equation}\label{eq72}
\sum_{i=L}^{2N-1}\left(i-N+\n-1\right)!\frac{1}{i!}\binom{i}{L}=
\frac{(L-N+\n-1)!}{L!}\binom{N+\n-1}{2N-1-L}\,.
\end{equation}
As before, setting $k=i-L$ and applying \cite[(5.9), p.~159]{GKP},
we see that the left hand side of \eqref{eq72} is
\begin{equation*}
\begin{split}
 \frac{1}{L!}\sum_{k=0}^{2N-1-L}&\left(L+k-N+\n-1\right)!\frac{1}{k!} \\*[5pt]
&=  \left(L-N+\n-1\right)! \sum_{k=0}^{2N-1-L}\binom{\n+L-N+k-1}{k}
\\*[5pt] &= \frac{(L-N+\n-1)!}{L!}\binom{N+\n-1}{2N-1-L}\,.
\end{split}
\end{equation*}

We are now ready to complete the proof of the lemma distinguishing
$4$~cases.

\vspace*{7pt}

\noindent {\sl Case 1:} $n$ odd.

Since $\beta(n,N)=0$, replacing in \eqref{eq70} $\alpha(n,N)$ by its
value and using \eqref{eq71} we get, by elementary arithmetics and
the identity $n\,V_n = \omega_n$,
\begin{equation*}
\begin{split}
A_L&=
%& (-1)^L\alpha(n,N)\sum_{i=L}^{2N-1}(N-\n)\cdots
%(N-\n-i+1)\frac{(-1)^i}{i!} {i\choose L} \\
 \frac{(-1)^L
\Gamma(2-\n)}{4^{N-1}(N-1)!\Gamma(N+1-\n)(2-n)\omega_n}(-1)^L
\binom{N-\n}{L} \binom{\n+N-1}{2N-1-L} \\*[7pt] &=
\frac{(-1)^{N+L}(N+\n-1)\cdots (\n+1)}{4^N \,V_n\,
L!\,(2N-1-L)!\,(L+\n-N)(N-1)!} \\*[7pt] &=
\frac{(-1)^{N+L}\binom{N+\n-1}{N-1}}{4^N \,V_n\,
L!\,(2N-1-L)!\,(L+\n-N)}\,.
\end{split}
\end{equation*}

\vspace*{7pt}

\noindent {\sl Case 2:} $n$ even, $n\neq 2$ and $N\le \n-1$.

As in case 1 $\beta(n,N)=0$, and we proceed similarly using
\eqref{eq71} to obtain
\begin{equation*}
\begin{split}
A_L&=
\frac{(-1)^L(-1)^{N-1}(\n-N-1)!}{4^N(N-1)!(\n-2)!(2-n)\omega_n}(-1)^L
\binom{N-\n}{L}\binom{N+\n-1}{2N-1-L}\\*[7pt] &=
\frac{(-1)^{N+1}\binom{N+\n-1}{N-1}\n
!(\n-N-1)!\{(N-\n)\cdots(N-\n-L+1)\}}{(2N-1-L)!\,
L!\,4^{N-1}\,(2-n)\,\omega_n \,(\n-2)!\,(\n-N+L)!}\\*[7pt] &=
\frac{(-1)^N\binom{N+\n-1}{N-1}}{4^N \,V_n\,
L!\,(2N-1-L)!}\,\frac{(-1)L(\n-N+L-1)!}{(\n-N+L)!}\\*[7pt] &=
\frac{(-1)^{N+L}\binom{N+\n-1}{N-1}}{4^N \,V_n\,
L!\,(2N-1-L)!\,(L+\n-N)}\,.
\end{split}
\end{equation*}

\vspace*{7pt}

\noindent {\sl Case 3:} $n$ even, $n\neq 2$ and $N\ge \n.$

Replacing in \eqref{eq70} $\alpha(n,N)$ and $\beta(n,N)$ by their
values and using \eqref{eq72} we get, by elementary arithmetics and
the identity $n\,V_n = \omega_n$,
\begin{equation*}
\begin{split}
A_L&=\beta(n,N)(N-\n)!(-1)^{N+\n-1+L}\sum_{i=L}^{2N-1}(i-N+\n-1)!\frac{1}{i!}
\binom{i}{L} \\*[7pt]
&=\frac{(-1)^{N+L}(L-N+\n-1)!}{L!4^{N-1}(\n-2)!(N-1)!(2-n)\omega_n}\binom{N+\n-1}{2N-1-L}\\*[7pt]
&=  \frac{(-1)^{L+N}\n(\n-1)}{4^{N-1}(2-n)\omega_n
L!(2N-1-L)!(\n-N+L)}\binom{N+\n-1} {N-1}   \\*[7pt] &=
\frac{(-1)^{L+N}\binom{N+\n-1}{N-1}}{4^NL!V_n(2N-1-L)!(\n-N+L)}\,.
\end{split}
\end{equation*}

\vspace*{7pt}

\noindent {\sl Case 4:} $n=2.$

Proceeding  as in case 3 and we obtain
\begin{equation*}
\begin{split}
A_L&=\beta(n,N)(N-\n)!(-1)^{N+\n-1+L}\sum_{i=L}^{2N-1}(i-N+\n-1)!\frac{1}{i!}
\binom{i}{L}\\*[7pt] &=
\frac{(-1)^{N+L}N!}{2\omega_24^{N-1}L!(N-1)!(2N-1-L)!(L+1-N)}\\*[7pt]
&= \frac{(-1)^{L+N}\binom{N}{N-1}}{V_24^NL!(2n-1-L)!(L+1-N)}\,.
\end{split}
\end{equation*}
\end{proof}

\begin{proof}[Proof of Lemma \ref{L10}]
Recall that (see \eqref{eq55bis})
$$
C_{2j} = \sum_{l=N+j}^{2N-1} c_{\,l,\,j,\,l-N-j}\,G_{\n+l-N+j}(0)\,.
$$
Thus, we have to compute the constants $c_{l,j,k}$ appearing in the
expression~\eqref{eq54tris} for~$\widehat{S\,\chi_{B}}(\xi)$. For
that we need the constants $c_{l,j}$ appearing in the formula
\eqref{eq21} for $S(x)$. We start by computing
$P_{2j}(\partial)\Delta^{N-j} (|x|^{2l})$. Using \eqref{eq19}
 and Lemma 4 one gets
\begin{equation*}
\begin{split}
P_{2j}(\partial)\Delta^{N-j} (|x|^{2l}) = \frac{4^N l!
(N-j)!}{(l-N-j)!}\binom{l-1+\n}{N-j} P_{2j}(x) |x|^{2(l-N-j)}
\end{split}
\end{equation*}
if $l-N-j\ge 0$ (and $=0$ if $l-N-j < 0$).

As in  \eqref{eq24} (Section 3), we want to compute
$P_{2j}(\partial)\Delta^{l-N-j}G_{\n}(\xi)$ by using Lemma \ref{L2}
applied to $f(r)=G_\n(r)$ and the homogeneous polynomial $L(x)=
P_{2j}(x)\,|x|^{2(l-N-j)}$. We obtain
\begin{equation*}
\begin{split}
P_{2j}(\partial)\Delta^{l-N-j}G_{\n}(\xi) &  =
(-1)^{2(l-N)}\sum_{k\ge 0}\frac{(-1)^k}{2^k k!}\Delta^{k} (
P_{2j}(x) | x|^{2(l-N-j}  )\, G_{\n+2(l-N)-k} (\xi)\\*[7pt]
 & = \sum_{k=0}^{l-N-j} \frac{(-1)^k}{2^k k!}  4^{k}\frac{ (l-N-j) !}{(l-N-j-k)!} k!
\binom{\n+j+l-N-1}{k} \\*[7pt]
 & \quad \times  P_{2j}(\xi) | \xi|^{2(l-N-j-k)}  \, G_{\n+2(l-N)-k}(\xi) \,.
\end{split}
\end{equation*}

In view of the definitions of $Q(x)$ and $S(x)$,
\begin{equation*}
\begin{split}
S(x) & = -Q(\partial)\left (  \sum_{l=0}^{2N-1}A_{l}|x|^{2l} \right
) = -  \sum_{l=0}^{2N-1}A_{l} \sum_{j=1}^N \gamma_{2j} P_{2j }
(\partial) \Delta^{N-j}(|x|^{2l}) \\*[7pt] & = - \sum_{l=N+1}^{2N-1}
\sum_{j=1}^{l-N} A_{l} \gamma_{2j} \frac{4^N l!
(N-j)!}{(l-N-j)!}\binom{l-1+\n}{N-j} P_{2j}(x)
|x|^{2(l-N-j)}\\*[7pt] & = \sum_{l=N+1}^{2N-1} \sum_{j=1}^{l-N}
c_{l,j}P_{2j}(x) |x|^{2(l-N-j)}\,,
\end{split}
\end{equation*}
where the last identity defines the $c_{l,j}$. In \eqref{eq54tris}
we set
\begin{equation*}
\begin{split}
\widehat{S\,\chi_{B}}(\xi) & = S(\imath\, \partial)\,
\widehat{\chi_B}(\xi)\\*[7pt] & =  \sum_{l=N+1}^{2N-1}
\sum_{j=1}^{l-N}
c_{lj}\,(-1)^{l-N}\,\,P_{2j}(\partial)\,\triangle^{l-N-j}\,G_{\n}(\xi)\\*[7pt]
& = \sum_{l=N+1}^{2N-1} \sum_{j=1}^{l-N} \sum_{k=0}^{l-N-j}
c_{\,l,\,j,\,k}\,\,P_{2j}(\xi)\,|\xi|^{2(l-N-j-k)}
G_{\n+2l-2N-k}(\xi)\,.
\end{split}
\end{equation*}
Consequently,
\begin{equation*}
\begin{split}
c_{l,j,k} & =c_{l,j} (-1)^{l-N} \frac{(-1)^k}{2^k}  4^{k}\frac{
(l-N-j) !}{(l-N-j-k)!} \binom{\n+j+l-N-1}{k} \\*[7pt] & =
-(-1)^{l+k+N}A_{l} \gamma_{2j} \frac{4^N l!
(N-j)!}{(l-N-j)!}\binom{l-1+\n}{N-j}\\*[7pt] &\quad\times
2^{k}\frac{ (l-N-j) !}{(l-N-j-k)!} \binom{\n+j+l-N-1}{k}\,.
\end{split}
\end{equation*}
Replacing $A_l$ by the formula given in lemma \ref{L13} and
performing some easy arithmetics we get
\begin{equation}\label{eq74}
c_{l,j,k}= - \frac{(-1)^{k}}{V_{n}}\frac{\gamma_{2j}\!\displaystyle
\binom{N+\n-1}{N-1}  2^k (N-j)! \binom{l-1+\n}{N-j}\!
\binom{\n+j+l-N-1} {k} } {(l+\n-N)(2N-l-1)!  (l-N-j-k)! } \, .\!
\end{equation}

\pagebreak

The final computation of the $C_{2j}$ is as follows.
{\allowdisplaybreaks
\begin{align*}
%\begin{split}
C_{2j} &  = \sum_{l=N+j}^{2N-1} c_{\,l,\,j,\,l-N-j}\,G_{\n+l-N+j}(0)\\[9pt]
& = \hspace{1.5cm} \text{[by the explicit value of $G_p(0)$ given in
\eqref{eq75} below]}\\[9pt]
& =  \sum_{l=N+j}^{2N-1} c_{\,l,\,j,\,l-N-j}
\frac{1}{2^{\n+l-N+j}\Gamma
(\n+l-N+j+1)}\\[9pt]
& = \hspace{1.5cm} \text{[by \eqref{eq74}]}\\[9pt]
& =-  \sum_{l=N+j}^{2N-1}\! \frac{(-1)^{l-N-j}}{V_{n}}
\frac{\gamma_{2j}\!\displaystyle\binom{N\!+\!\n\!-\!1}{N\!-\!1}  2^{l-N-j} (N\!-\!j)! \binom{l\!-\!1\!+\!\n}{N-j} \binom{\n\!+\!j\!+\!l\!-\!N\!-\!1}{l-N-j} } {(l+\n-N)(2N-l-1)!  2^{\n+l-N+j}\Gamma(\n+l-N+j+1)} \\[9pt]
&= \hspace{1.5cm} \text{[substituting  the value given in
\eqref{eq7} in
$\gamma_{2j}$]}\\[9pt]
& =- \frac{ \pi^\n \Gamma(j) \displaystyle\binom{N+\n-1}{N-1} (N-j)! }{ V_{n} \Gamma(j+\n) 2^{\n+2j}  }\\[9pt]
&\hspace{1.5cm}\times
\sum_{l=N+1}^{2N-1}\frac{(-1)^{l+N}\displaystyle \binom{l-1+\n}{N-j}
\binom{\n+j+l-N-1}{l-N-j} } {(l+\n-N)(2N-l-1)!  \Gamma(\n+l-N+j+1)}\\[9pt]
&=  \hspace{1.5cm}  [\text{setting }  l=i+N+j]  \\[9pt]
& =- \frac{ \pi^\n \Gamma(j) \displaystyle\binom{N+\n-1}{N-1} (N-j)! }{ V_{n} \Gamma(j+\n) 2^{\n+2j}  }  \\[9pt]
& \hspace{1.5cm}\times \sum_{i=0}^{N-1-j}\frac{(-1)^{i+j}\displaystyle \binom{i+N+j-1+\n}{N-j} \binom{\n+2j+i-1}{i} } {(i+j+\n)(N-i-j-1)!  \Gamma(\n+i+2j+1)}\\[9pt]
&  = \hspace{1.5cm} [\text{because }   \Gamma(\n+i+2j+1)= \Gamma(2j +\n) \prod_{k=0}^{i} (\n+2j+i-k) ]\\[9pt]
& =- \frac{ (-1)^{j}\pi^\n \Gamma(j) \displaystyle\binom{N+\n-1}{N-1} (N-j)! }{ V_{n} \Gamma(j+\n) 2^{\n+2j} \Gamma(2j +\n) }\\[9pt]
& \hspace{1.5cm} \times \sum_{i=0}^{N-1-j}\frac{(-1)^{i}
\displaystyle \binom{i+N+j-1+\n}{N-j} \binom{\n+2j+i-1}{i} }
 {(i+j+\n)(N-i-j-1)!  \prod_{k=0}^{i} (\n+2j+i-k)}\\[9pt]
& =  \hspace{1.5cm} [\text{using Lemma \ref{sublema2} below}]\\[9pt]
&=- \frac{ (-1)^{j}\pi^\n \Gamma(j) \displaystyle\binom{N+\n-1}{N-1} (N-j)! }{ V_{n} \Gamma(j+\n) 2^{\n+2j} \Gamma(2j +\n) } \binom{N-1}{j-1}\frac{\Gamma (j+\n)}{j\, \Gamma (N+\n)}\\[9pt]
& = \frac{- \pi^\n}{V_{n} 2^\n \Gamma(\n+1)} \; \frac{(-1)^j }{j\,
4^{j} \Gamma (2j+ \n)} \,.
%\end{split}
\end{align*}}
\end{proof}

\begin{lemma}\label{sublema2}
For each $j=1, \dotsc , N-1$
$$
\sum_{i=0}^{N-1-j}\frac{(-1)^{i}\displaystyle
\binom{i+N+j-1+\n}{N-j} \binom{\n+2j+i-1} {i} } {(i+j+\n)(N-i-j-1)!
\prod_{k=0}^{i} (\n+2j+i-k)} = \binom{N-1}{j-1}\frac{\Gamma
(j+\n)}{j \Gamma (N+\n)}\,.
$$
\end{lemma}

\begin{proof}
Divide the left hand side by the right hand side and denote the
quotient by~$A$. We have to prove that $A=1$.
%$$
%A= \sum_{i=0}^{N-1-j}\frac{(-1)^{i}\displaystyle {i+N+j-1+\n \choose
%N-j} {\n+2j+i-1 \choose i}  j \Gamma (N+\n)} {(i+j+\n)(N-i-j-1)!
%\prod_{k=0}^{i} (\n+2j+i-k) {N-1\choose j-1} \Gamma (j+\n)} =1
%$$
Using elementary arithmetics
\begin{multline*}
\frac{\displaystyle \binom{i+N+j-1+\n}{N-j}
\binom{\n+2j+i-1}{i}\Gamma(N+\n) }{(i+j+\n)  \prod_{k=0}^{i}
(\n+2j+i-k)\; \Gamma (j+\n)} \\*[9pt] =
\binom{N+i+j+\n-1}{N-j-1}\binom{N+\n-1}{N-i-j-1}\frac{(N-i-j-1)!}{N-j}
\binom{\n+i+j-1} {i}\,,
\end{multline*}
and so
\begin{equation*}
\begin{split}
A &=  \frac{j}{\binom{N-1}{j-1}(N-j)}\! \sum_{i=0}^{N-1-j}\!
(-1)^{i}\!
\binom{N+i+j+\n-1}{N-j-1}\!\binom{N+\n-1}{N-i-j-1}\!\binom{\n+i+j-1}{i}\\*[7pt]
& = \hspace{1.5cm} \left[\text{because } \binom{a+i}{i}=
(-1)^{i}\binom{-a-1}{i} \right]\\*[7pt] &=
\frac{j}{\binom{N-1}{j-1}(N-j)} \sum_{i=0}^{N-1-j}
\binom{N+i+j+\n-1}{N-j-1}\binom{N+\n-1}{N-i-j-1}\binom{-\n-j}{i}\\*[7pt]
& =  [\text{by the triple-binomial identity (5.28) of
(\cite[p.~171]{GKP}), see below}]\\*[7pt] & =
\frac{j}{\binom{N-1}{j-1}(N-j)}
\binom{\n+N+j-1}{0}\binom{N-1}{N-j-1} =1 \,.
\end{split}
\end{equation*}
For the reader's convenience and later reference we state the
triple-binomial identity~\cite[(5.28), p.~171]{GKP}:
\begin{equation}\label{eq74bis}
\sum _{k=0}^{n} \binom{m-r+s}{k}\binom{n+r-s}{n-k}\binom{r+k}{m+n} =
\binom{r}{m}\binom{s}{n}\,,
\end{equation}
where $m$ and $n$ are non-negative integers.
\end{proof}

Our next task is to prove Lemma \ref{L11} and Lemma \ref{L12}.
Setting $\xi= r\,\xi_0$, $|\xi_{0}|=1$, in \eqref{eq54tris} we
obtain
\begin{equation*}
\begin{split}
\widehat{S_{N}\,\chi_{B}}(r\xi_{0}) & = \sum_{l=N+1}^{2N-1}
\sum_{j=1}^{l-N} \sum_{k=0}^{l-N-j} c_{\,
l,\,j,\,k}\,\,P_{2j}(r\xi_{0})\,\,|r\xi_{0}|^{2(l-N-j-k)}
G_{\n+2l-2N-k}(r\xi_{0})\\*[5pt] & \hspace{1.5 cm} [\text{make the
change of indexes  $l=N+s$}]\\*[5pt] & = \sum_{s=1}^{N-1}
\sum_{j=1}^{s} \sum_{k=0}^{s-j}
c_{\,N+s,\,j,\,k}\,P_{2j}(\xi_{0})\,r^{2(s-k)}
G_{\n+2s-k}(r)\\*[5pt] & = \sum_{j=1}^{N-1} \sum_{s=j}^{N-1}
\sum_{k=0}^{s-j} c_{\,N+s,\,j,\,k}\,P_{2j}(\xi_{0})\,r^{2(s-k)}
G_{\n+2s-k}(r)\\*[5pt] & := \sum_{p=1}^{\infty} a_{2p}^{N}(\xi_{0})
r^{2p}\,.
\end{split}
\end{equation*}

In order to compute the coefficients $ a_{2p}^{N}(\xi_{0})$ we
substitute the power series expansion of $G_{q}(r)$ \cite[A-8]{Gr},
namely,
\begin{equation}\label{eq75}
G_{q}(r)=\sum _{i=0}^{\infty}\frac{(-1)^{i}}{i!\, \Gamma(q+i+1)}
\frac{r^{2i}}{2^{2i+q}}\,,
\end{equation}
in the last triple sum above.

\begin{proof}[Proof of Lemma \ref{L11}]
We are assuming that $1 \le p \le N-1$. It is important to remark
that, for this range of $p$, after introducing \eqref{eq75} in the
triple sum above,  only the values of the index~$j$ satisfying $1
\le j \le p$ are involved in the expression for $a_{2p}^N$. Once
\eqref{eq75} has been introduced in the triple sum one should sum,
in principle, on the four indexes $i$, $j$, $s$ and $k$. But since
we are looking at the coefficient of $r^{2p}$ we have the relation
$2(s-k)+2i=2p$, which actually leaves us with three indexes. The
range of each of these indexes is easy to determine and one gets
$$
a_{2p}^N  = \sum_{j=1}^{p}P_{2j}(\xi_{0}) \sum_{i=0}^{p-j}
\sum_{s=p-i}^{N-1} c_{\,N+s,\,j,\,s-(p-i)}
 \times \text{coefficient of $r^{2i}$ from $G_{\n+s+p-i}(r)$}\,.
$$
In view of \eqref{eq75}
\begin{equation*}
\begin{split}
a_{2p}^N &= \sum_{j=1}^{p}P_{2j}(\xi_{0}) \sum_{i=0}^{p-j}
\sum_{s=p-i}^{N-1} c_{\,N+s,\,j,\,s-(p-i)} \frac{(-1)^{i}}{i!
2^{i+\n+s+p}\Gamma(\n+s+p+1)}\\*[7pt] & = \hspace{1.5cm}[\text{by
the expression \eqref{eq74} for $c_{l,j,k}$}]\\*[7pt] & =
\sum_{j=1}^{p}P_{2j}(\xi_{0}) \sum_{i=0}^{p-j} \sum_{s=p-i}^{N-1}
\frac{(-1)^{i+1}}{i! 2^{i+\n+s+p}\Gamma(\n+s+p+1)}
\frac{(-1)^{s-(p-i)}}{V_{n}} \\*[7pt] &  \hspace{1.5cm}\times
\frac{\gamma_{2j}\displaystyle \binom{N+\n-1}{N-1}  2^{s-(p-i)}
(N-j)! \binom{N+s-1+\n}{N-j} \binom{\n+j+s-1}{s-(p-i)} }
{(s+\n)(N-s-1)!  (p-i-j)! }\\*[7pt] & =
(-1)^{p+1}\sum_{j=1}^{p}P_{2j}(\xi_{0})  \frac{(-1)^{j} \pi^\n
\Gamma(j) (N-j)!}{V_{n} \Gamma(\n+j)2^{\n+2p}}\sum_{i=0}^{p-j}
\frac{\displaystyle \binom{N+\n-1}{N-1} }{i! (p-i-j)! }\\*[7pt] &
\hspace{1.5cm}\times  \sum_{s=p-i}^{N-1} \frac{(-1)^s\displaystyle
\binom{N+s-1+\n}{N-j} \binom{\n+j+s-1}{ s-(p-i)} } {(s+\n)(N-s-1)!
\Gamma(\n+s+p+1)}\,.
\end{split}
\end{equation*}

In Lemma \ref{L15} below we give a useful compact form for the last
sum. Using it we obtain
\begin{multline*}
a_{2p}^N  = (-1)^{p+1}\sum_{j=1}^{p}P_{2j}(\xi_{0}) \frac{(-1)^{j}
\pi^\n \Gamma(j) (N-j)!}{V_{n} \Gamma(\n+j)2^{\n+2p}}
\sum_{i=0}^{p-j} \frac{\displaystyle \binom{N+\n-1}{N-1} }{i!
(p-i-j)! }\\*[7pt] \times \frac{ (-1)^{p-i}(N-p-1)!
\displaystyle\binom{N-1}{p}}{(N-j)! \Gamma(N+\n)
j!\binom{\n+p-i+j-1}{j}}\,.
\end{multline*}
Easy arithmetics with binomial coefficients gives
$$
\binom{N+\n-1}{N-1}(N-p-1)!\binom{N-1}{p}\frac{1}{\Gamma(N+\n)}=
\frac{1}{p! \, \Gamma(\n+1)}\,.
$$
We finally get the extremely surprising identity
\begin{equation}\label{eq78}
 a_{2p}^N \!=\! \frac{-\pi^\n}{V_{n}
2^{\n+2p}p!\Gamma(\n+1)} \!\sum_{j=1}^{p}\! \frac{(-1)^j \Gamma(j)
P_{2j}(\xi_{0})} {\Gamma(\n+j) } \sum_{i=0}^{p-j} \frac{(-1)^{i}}{
i! (p-i-j)! j! \binom{\n+p-i+j-1}{j}}\,,
\end{equation}
in which $N$ has miraculously disappeared.  Thus Lemma \ref{L11} is
proved.
\end{proof}

\begin{proof}[Proof of Lemma \ref{L12}]
We start by proving the inequality \eqref{eq59bis}, so that $1\le p
\le N-1$.  We roughly estimate  $a_{2p}= a_{2p}^N$ by putting the
absolute value inside the sums in \eqref{eq78}. The absolute value
of each term in the innermost sum in \eqref{eq78} is obviously not
greater than~$1$ and there are at most $p$~terms. The factor in
front of $P_{2j}(\xi_0)$ is again not greater than~$1$ in absolute
value. Denoting by $C$ the terms that depend only on $n$ we obtain
the desired inequality \eqref{eq59bis}.

We turn now to the proof of inequality \eqref{eq59tris}. Recall that
$$
 \widehat{S_{N}\,\chi_{B}}(r\xi_{0})  = \sum_{p=1}^{\infty}
a_{2p}^{N}(\xi_{0}) r^{2p}
 =  \sum_{j=1}^{N-1} \sum_{s=j}^{N-1} \sum_{k=0}^{s-j}
c_{\,N+s,\,j,\,k}\,\,P_{2j}(\xi_{0})\,\,r^{2(s-k)} G_{\n+2s-k}(r)\,.
$$
Replacing $G_{\n+2s-k}(r)$ by the expression given by \eqref{eq75}
we obtain, as before, a sum with four indexes. Now we eliminate the
index $i$ of  \eqref{eq75} using $s-k+i=p$. Hence
\begin{equation*}
\begin{split}
a_{2p}^N & = \sum_{j=1}^{N-1}P_{2j}(\xi_{0}) \sum_{s=j}^{N-1}
\sum_{k=0}^{s-j} c_{\,N+s,\,j,\, k}
 \times \text{coefficient of $r^{2(p-s+k)}$ from $G_{\n+2s-k}(r)$}\\*[7pt]
&=\sum_{j=1}^{N-1}P_{2j}(\xi_{0}) \sum_{s=j}^{N-1} \sum_{k=0}^{s-j}
c_{\,N+s,\,j,\, k}
\frac{(-1)^{p-s+k}}{(p-s+k)!\Gamma(\n+p+s-1)2^{2p+\n+k}}\\*[7pt]
&=\sum_{j=1}^{N-1}P_{2j}(\xi_{0}) \sum_{k=0}^{N-1-j}
\sum_{s=j+k}^{N-1} c_{\,N+s,\,j,\, k}
\frac{(-1)^{p-s+k}}{(p-s+k)!\Gamma(\n+p+s-1)2^{2p+\n+k}}\\*[7pt] &=
\frac{(-1)^{p}}{V_{n}4^{p}2^\n}\binom{N+\n-1}{N-1}\sum_{j=1}^{N-1}
\gamma_{2j}P_{2j}(\xi_{0}) \sum_{k=0}^{N-1-j} \sum_{s=j+k}^{N-1}
\\*[7pt] &\quad\times  \frac{(-1)^s (N-j)!
\displaystyle\binom{N+s-1+\n}{N-j}\binom{\n+j+s-1}{k}}{(p-s+k)!
\Gamma(\n+p+s+1) (s+\n) (N-s-1)! (s-j-k)!}\,.
\end{split}
\end{equation*}
The second identity is just \eqref{eq75}. The third is a change of
the order of summation and the latest follows from the formula
\eqref{eq74} for the constants $c_{l,j,k}$.

In view of the elementary fact that
$$
(N-j)! \displaystyle\binom{N+s-1+\n}{N-j}\binom{\n+j+s-1}{k} =
\frac{\Gamma(s+\n+N)}{k! \Gamma(s+\n+j-k)}
$$
we get
\begin{equation*}
\begin{split}
&\left| \sum_{k=0}^{N-1-j} \sum_{s=j+k}^{N-1}  \frac{(-1)^s (N-j)!
\displaystyle\binom{N+s-1+\n}{N-j}\binom{\n+j+s-1}{k}}{(p-s+k)!
\Gamma(\n+p+s+1) (s+\n) (N-s-1)! (s-j-k)!} \right | \\*[7pt] &\le
\sum_{k=0}^{N-1-j}\frac{1}{k!}\\*[7pt]
&\quad\times\sum_{s=j+k}^{N-1}\frac{1}{\Gamma(s+\n+j-k)(p-s+k)!
(\n+p+s) (s+\n) (N-s-1)! (s-j-k)!}\\*[7pt] &\le \sum_{k=0}^{N-1-j}
\frac{1}{k!}\sum_{s=j+k}^{N-1}\frac{1}{ (s-j-k)!} \le e^2 \,,
\end{split}
\end{equation*}
where in the first inequality we used that, since $N \le p$,
$$
\frac{\Gamma(s+\n+N)}{\Gamma(\n+p+s+1)} \le \frac{1}{\n+p+s}\,.
$$
The proof of \eqref{eq59tris} is complete.
\end{proof}

\begin{lemma}\label{L15}
Let $N-1\ge L\ge j\ge k\ge 0$ be integers. Then
\begin{multline*}
\sum_{s=j}^{N-1} \frac{(-1)^s\binom{\n+N+s-1}{N-k}\binom{\n+k+s-1}
{s-j}}{(s+\n)(N-s-1)!\Gamma(\n+s+L+1)} \\*[7pt] =
(-1)^{j}\frac{(N-L-1)!}{(N-k)!}\binom{N-1}{L}\frac{1}{k!\,
\Gamma(\n+N)}\frac{1}{\binom{\n+k+j-1}{k}}\,.
\end{multline*}
\end{lemma}

\begin{proof} To simplify notation set
$$
A = \sum_{s=j}^{N-1}
\frac{(-1)^s\binom{\n+N+s-1}{N-k}\binom{\n+k+s-1}
{s-j}}{(s+\n)(N-s-1)!\Gamma(\n+s+L+1)}\,.
$$
Making the change of index of summation $i=s-j$ and using repeatedly
the identity~$\Gamma(x+1)=x\,\Gamma(x)$ we have
\begin{equation*}
\begin{split}
 A &=  \frac{(-1)^{j}}{\Gamma(\n+j+L)}\sum_{i=0}^{N-1-j}
\frac{(-1)^{i}\binom{\n+N+i+j-1}{N-k}\binom{\n+k+i+j-1
}{i}}{(i+j+\n)(N-i-j-1)!\{\prod_{p=0}^{i}(\n+i+j+L-p)\}}\\*[7pt]
 &= \frac{(-1)^{j}}{\Gamma(\n+j+L)}\; B \,,
\end{split}
\end{equation*}
where the last identity defines $B$. To compute $B$ we need to
rewrite the terms in a more convenient way so that we may apply well
known equalities, among which the triple-binomial
identity~\eqref{eq74bis}. Notice that
\begin{multline*}
 \frac{\binom{\n+N+i+j-1}{N-k}\binom{\n+k+i+j-1}{i}}{\{\prod_{p=0}^{i}(\n+i+j+L-p)\}}
 \\*[7pt]
 =\frac{(N-L-1)!(L-k)!}{i!(N-k)!} \binom{\n+N+i+j-1}{N-L-1}\binom{\n+j+L-1}{L-k}\,.
\end{multline*}
Hence
$$
B\!=\!\sum_{i=0}^{N-1-j}\!\frac{(-1)^i(N\!-\!L\!-\!1)!(L-k)!}{i!(N\!-\!k)!(N\!-\!i\!-\!j\!-\!1)!(i+j+\n)}\binom{\n+N+i+j-1}{N-L-1}\!\binom{\n+j+L-1}{L-k}\,.
$$
Since clearly
$$
 \frac{\Gamma(\n+N)}{\Gamma(\n+j)} \frac{1}{(i+j+\n)(N-i-j-1)!}
=  \binom{\n+N-1}{N-j-i-1}\binom{\n+j+i-1}{i}i! \,,
$$
$B$ can be written as
\begin{equation*}
\begin{split}
B&=
\frac{\Gamma(\n+j)(N-L-1)!(L-k)!}{\Gamma(\n+N)(N-k)!}\binom{\n+j+L-1}{L-k}
\\*[7pt]
&\quad\times\sum_{i=0}^{N-1-j}(-1)^i\binom{\n+N-1}{N-j-i-1}\binom{\n+j+i-1}{i}\binom{\n+N+i+j-1}
{N-L-1} \\*[7pt] &=
\frac{\Gamma(\n+j)(N-L-1)!(L-k)!}{\Gamma(\n+N)(N-k)!}\binom{\n+j+L-1}
{L-k}\;C \,,
\end{split}
\end{equation*}
where the last identity defines $C$.

 To simplify notation set
$L=m+j,$ where $0\le m\le L-j.$ Using the elementary identity
$$
\binom{\n+N-1}{N-j-i-1}\binom{\n+i+j-1}{i}
=\frac{(N-j)!}{i!(N-i-j-1)!(\n+i+j)}\binom{\n+N-1}{N-j}\,,
$$
we get
$$
C  =\binom{\n+N-1}{N-j}(N-j)!\sum_{i=0}^{N-1-j}\frac{(-1)^i
\binom{\n+N+i+j-1}{N-j-m-1}}{i!\,(N-i-j-1)!\,(\n+i+j)}\,.
$$
The only task left is the computation of the sum
$$
D(j,m)= \sum_{i=0}^{N-1-j}\frac{(-1)^i
\binom{\n+N+i+j-1}{N-j-m-1}}{i!\,(N-i-j-1)!\,(\n+i+j)}\,.
$$
The identity
$$
\frac{1}{\n+i+j}=\frac{1}{\n+j}\left(1-\frac{i}{\n+i+j}\right)\,,
$$
yields the expression
\begin{equation*}
\begin{split}
D(j,m)&
=\sum_{i=0}^{N-1-j}\frac{(-1)^i}{(\n+j)i!(N-i-j-1)!}\binom{\n+i+j+N-1}{N-j-m-1}
\\*[7pt] &\quad -\sum_{i=0}^{N-1-j}\frac{(-1)^i
i}{i!(\n+j)(\n+j+i)(N-i-j-1)!}\binom{\n+i+j+N-1}{N-j-m-1}\,.
\end{split}
\end{equation*}
The first sum in the above expression for $D(j,m)$ turns out to
vanish for $m \ge 1$. This is because
\begin{equation*}
\begin{split}
\sum_{i=0}^{N-1-j}&\frac{(-1)^i}{(\n+j)i!(N-i-j-1)!}\binom{\n+i+j+N-1}
{N-j-m-1} \\*[7pt]
 &= \frac{1}{(N-j-1)!\;(\n+j)}\sum_{i=0}^{N-1-j}(-1)^i\binom{N-j-1}
{i}\binom{\n+i+j+N-1}{N-j-m-1}\\*[7pt]
  & = (-1)^{N-j-1} \binom{\n+j+N-1}{-m}\\*[7pt]
& =0 \,,
\end{split}
\end{equation*}
where the next to the last equality follows from an identity proven
in \cite[(5.24), p.~169]{GKP} and the last equality follows from the
fact that $\displaystyle\binom{s}{k}=0$ provided $k$ is a negative
integer. Hence, setting $s=i-1$,
\begin{equation*}
\begin{split}
D(j,m) &=-\sum_{i=0}^{N-1-j}\frac{(-1)^i\;
i}{i!(\n+j)(\n+j+i)(N-i-j-1)!}\binom{\n+i+j+N-1}{N-j-m-1}\\*[7pt]
 & =\frac{1}{\n+j}\sum_{s=0}^{N-(j+1)-1}\frac{(-1)^s}{s!(\n+s+(j+1))}\frac{\binom{\n+(j+1)+N+s-1}{N-(j+1)-(m-1)-1}}{(N-(j+1)-s-1)!} \\*[7pt]
& = \frac{1}{(\n+j)}D(j+1,m-1)\,.
\end{split}
\end{equation*}
Repeating  the above argument $m$ times we obtain that
$$
D(j,m)=\frac{1}{(\n+j)(\n+j+1)\cdots (\n+j+m-1)}\;D(L,0)\,.
$$
To compute $D(L,0)$ we use the elementary identity
$$
\frac{1}{(\n+L+s)}=\frac{s!(N-L-s-1)!\Gamma(\n+L)}{\Gamma(\n+N)}\binom{\n+L+s-1}
{s}\binom{\n+N-1}{N-L-s-1}\,,
$$
from which we get
\begin{equation*}
\begin{split}
D(L,0)&=\sum_{s=0}^{N-1-L}\!(-1)^s\frac{\Gamma(\n+L)}{\Gamma(\n+N)}\!\binom{\n+L+s-1}{s}\!\binom{\n+N-1}{N\!-\!L\!-\!s\!-\!1}\!\binom{\n\!+\!L\!+\!N\!+\!s\!-\!1}{N-L-1}\\*[7pt]
&
=\frac{\Gamma(\n+L)}{\Gamma(\n+N)}\sum_{s=0}^{N-1-L}\binom{-\n-L}{s}\binom{\n+N-1}{N-L-s-1}\binom{\n+L+N+s-1}{N-L-1}\\*[7pt]
&=\frac{\Gamma(\n+L)}{\Gamma(\n+N)}\binom{N-1}{L}\,,
\end{split}
\end{equation*}
where in the second identity we applied \cite[(5.14), p.~164]{GKP}
and the latest equality is consequence of the triple-binomial
identity \cite[(5.28), p.~171]{GKP} (for $ n=N-L-1$, $m=0$,
$r=\n+N+L-1$ and $s=N-1$). Consequently,
\begin{equation*}
\begin{split}
 D(j,m)& =\frac{\Gamma(\n+L)}{\Gamma(\n+N)}\binom{N-1}{L}\frac{1}{(\n+j)\cdots (\n+L-1)}\\*[7pt]
&=\frac{\Gamma(\n+j)}{\Gamma(\n+N)}\binom{N-1}{L}\,.
\end{split}
\end{equation*}
Hence
$$
C=\binom{\n+N-1}{N-j}\;(N-j)!\;
\frac{\Gamma(\n+j)}{\Gamma(\n+N)}\binom{N-1}{L}\,,
$$
and
\begin{multline*}
B=
\frac{\Gamma(\n+j)}{\Gamma(\n+N)}\frac{(N-L-1)!(L-k)!}{(N-k)!}\binom{\n+j+L-1}{L-k}\binom{\n+N-1}{N-j}\\*[7pt]
 \times \binom{\n+N-1}{N-j}\;(N-j)!\;\frac{\Gamma(\n+j)}{\Gamma(\n+N)}\binom{N-1}{L}\,.
\end{multline*}
Finally, after appropriate simplifications,
\begin{equation*}
\begin{split}
A&= \frac{(-1)^{j}}{\Gamma(\n+j+L)}\;B \\*[7pt] &=
(-1)^{j}\frac{(N-L-1)!}{(N-k)!}\frac{\binom{N-1}{L}}{\Gamma(\n+N)k!\binom{\n+j+k-1}
{k}}\,,
\end{split}
\end{equation*}
which completes the proof of the lemma.
\end{proof}

\section{Examples and questions}

\begin{example}
 Consider the polynomial operator $T=T_\lambda$  in $\mathbb{R}^2$
determined by
$$
\Omega(x,y) = \frac{xy}{|z|^2}+\lambda \frac{x^3y-x y^3}{|z|^4}\,,
$$
where $z= x+i y$. We claim that the inequality $
 \| T_\lambda^{\star} f  \|_2 \leq  C \| T_\lambda f \|_2$, $f \in L^2(\mathbb{R}^2)$,
holds if and only if $|\lambda|<2$. This follows from the Theorem
because the
 multiplier of the operator $T_\lambda$ is
 $$
 \frac{\xi\eta}{\xi^2 + \eta^2}\left(-\pi +\lambda\frac{\pi}{2} \frac{(\xi^2
 -\eta^2)}{\xi^2 + \eta^2}\right)\,,
 $$
$\frac{\xi\eta}{\xi^2 + \eta^2} $ is the multiplier of a second
order Riesz Transform and the function $ -\pi +\lambda\frac{\pi}{2}
\frac{(\xi^2  -\eta^2)}{\xi^2 + \eta^2}$ has no zeroes on the unit
circle if and only if $|\lambda| < 2$.
\end{example}

\begin{example}
Let $B$ stand for the Beurling transform and consider  the operator
$$
T_\lambda = B-\lambda\,B^2, \quad \lambda \in \C\,,
$$
which corresponds to
$$
\Omega(z)= -\frac{1}{\pi} \left(\frac{\overline{z}^2 |z|^2 +
2\,\lambda \, \overline{z}^4}{|z|^4}\right) \,.
$$
In this case the multiplier is
$$
\frac{\overline{\xi}^3 \xi \,-\lambda\,\overline{\xi}^4}{|\xi|^4}=
\frac{\overline{\xi}}{\xi} \left( 1 - \lambda \,
\frac{\overline{\xi}}{\xi}\right),\quad \xi\in \mathbb C \, ,
$$
which is the multiplier of $B$ times a function that vanishes on the
unit circle  if and only if $|\lambda|=1.$ Then the Theorem tells us
that $  \| T_\lambda^{\star} f  \|_2 \leq  C \| T_\lambda f \|_2$,
$f \in L^2(\mathbb{R}^2)$, if and only if $|\lambda|\ne 1$. Indeed,
strictly speaking, the necessary condition of the Theorem, i.e.\
(iii) implies (i), applies only to \emph{real} homogeneous
polynomials. However, in the case at hand the factorization results
one needs can be checked by direct inspection.
\end{example}

\begin{example}
We give an example of a polynomial operator $T$ which is of the
form~$T=R\circ U$, where $R$ is an even higher order Riesz Transform
and $U$ is invertible, but the $L^2$ inequality $ \| T^{\star} f
\|_2 \leq  C \| T f \|_2$, $f \in L^2(\Rn)$, does not hold. The
operator $T$ is associated with the homogenous polynomial of
degree~$8$
$$
P(x,y)=\frac{1}{\gamma_2}\,P_2(x,y)(x^2+y^2)^3+\varepsilon \left(
\frac{1}{\gamma_4}\,P_4(x,y)(x^2+y^2)^2
-\frac{1}{\gamma_8}\,P_8(x,y)\right)\,,
$$
where
\begin{equation*}
\begin{split}
P_2(x,y) &=xy \,,\\
P_4(x,y)&=x^4 -6x^2y^2 +y^4\,,\\
\intertext{and}
 P_8(x,y)&=x^8 +y^8 -28x^6 y^2 -28 x^2y^6+70x^4 y^4
\end{split}
\end{equation*}
are harmonic polynomials and $\varepsilon$ is  small enough. Notice
that $P_2$ does not divide $P_4$ nor $P_8$.  Therefore, by the
Theorem, there is no control of $T^{\star}$ in terms of $T$. On the
other hand, $P_4$ and $P_8$ have been chosen so that
$P_4(\xi_1,\xi_2)|\xi|^4 -P_8(\xi_1,\xi_2)$ is divisible by $P_2$,
and so the multiplier of $T$ is
$$
\frac{P_2(\xi_1,\xi_2)}{|\xi|^2}+\varepsilon\left(
\frac{P_4(\xi_1,\xi_2)}{|\xi|^4}
-\frac{P_8(\xi_1,\xi_2)}{|\xi|^8}\right) =
\frac{P_2(\xi_1,\xi_2)}{|\xi|^2} \left(1 +\varepsilon
\frac{Q(\xi_1,\xi_2)}{|\xi|^6} \right)\,,
$$
where $Q$ is a homogeneous polynomial of degree~$6$. Define $R$ as
the higher order Riesz transform whose multiplier is
$\frac{P_2(\xi_1,\xi_2)}{|\xi|^2}$ and  $U$ as the operator whose
multiplier is $1 +\varepsilon \frac{Q(\xi_1,\xi_2)}{|\xi|^6}$. If
$\varepsilon$ is small enough, then $U$ is invertible.
\end{example}

\begin{example}
We show here that condition (iii) in the Theorem is very
restrictive. Take $n=2$. Since $\Omega(e^{i\theta})$ is an even
function it has, modulo a rotation,  a Fourier series expansion of
the type
\begin{equation}\label{eq80}
\Omega(e^{i\theta}) = \sum_{j\geq 1} a_j\, \sin (2 j \theta)\,.
\end{equation}
Thus the harmonic polynomials $P_{2j}$ are
$$
P_{2j}(z) = a_j  \frac{z^{2j}-\overline{z}^{2j}}{2i}\,.
$$
If $a_j \ne 0$, then $P_{2j}$  vanishes exactly on $2j$ straight
lines through the origin uniformly distributed and containing the
two axis. Each such line is determined by a pair of opposed $4j$-th
roots of the unity. Assume now that the operator determined
by~$\Omega (e^{i\theta})$ satisfies condition $(iii)$ of the
Theorem. Let $j_0$ be the first positive integer with $a_{j_0} \ne
0$. Then only the $a_j$ with $j$ a multiple of $j_0$ may be
non-zero, owing to the particular structure of the zero set of
$P_{2j}$. In other words, the first non-zero $P_{2j}$ determines all
the others, modulo the constants $a_j$.
\end{example}

\begin{example}
We show now a method to construct even kernels in $\mathbb{R}^3$
that satisfy condition $(iii)$ in the Theorem. It can be easily
adapted to any dimension. The kernel is determined by the function
$$
\Omega(x,y,z)= xy \,\sum_{j\geq 0} \epsilon_j\,Q_{2j}(x,y,z)
$$
where the sequence $(\epsilon_j)$ is chosen so that $\Omega$ is in
$C^\infty(S^{2})$ and the $Q_{2j}$ are defined by
$$
Q_{2j}(x,y,z)= \sum_{k=0}^{2j} c_k\,y^{2k}\,z^{2j-2k}\,.
$$
The $c_k$ are determined by a recurrent formula obtained by
requiring that\linebreak $ xy\,Q_{2j}(x,y,z)$ is harmonic. Computing
its Laplacean we get the recurrent condition
$$
c_k= - c_{k-1}  \frac{2k (2k+1)}{(2j-2k+1)\,(2j-2k+2)},\quad 1\leq k
\leq j\,,
$$
where $c_0$ may be freely chosen.
\end{example}

We would like to close the paper by asking a couple of questions
which we have not been able to answer.

\begin{question}
Since our methods are very much dependent on the Fourier transform
we do not know whether either the weak type inequality
$$
\|T^\star f \|_{1,\infty} \leq C\| Tf \|_1 \,,
$$
or the $L^p$ inequality with $ 1<p< \infty$, $p \ne 2$,
$$
\| T^{\star} f  \|_p \leq  C \| T f \|_p ,\quad f \in L^p(\Rn)\,,
$$
imply the $L^2$ inequality
$$
\| T^{\star} f  \|_2 \leq  C \| T f \|_2,\quad f \in L^2(\Rn)\,.
$$
As suggested by Carlos P\'{e}rez,  this might be related to
interpolation results for couples of sub-linear operators.
\end{question}

\begin{question}
How far may the smoothness assumption on $\Omega$ be weakened ? More
concretely, does the Theorem still hold true for $\Omega$ of class
$C^m(S^{n-1})$ for some positive integer $m$ ?
\end{question}

\begin{gracies}
The authors were partially supported by grants\newline 2009SGR420
(Generalitat de Catalunya) and  MTM2010-15657 (Ministerio de Ciencia
e Innovaci\'{o}n).

The authors are indebted to F.~Ced\'{o} and X.~Xarles for several
illuminating conversations on the Division Lemma.
\end{gracies}

\begin{tabular}{l}
Joan Mateu\\
Departament de Matem\`{a}tiques\\
Universitat Aut\`{o}noma de Barcelona\\
08193 Bellaterra, Barcelona, Catalonia\\
{\it E-mail:} {\tt mateu@mat.uab.cat}\\ \\
Joan Orobitg\\
Departament de Matem\`{a}tiques\\
Universitat Aut\`{o}noma de Barcelona\\
08193 Bellaterra, Barcelona, Catalonia\\
{\it E-mail:} {\tt orobitg@mat.uab.cat}\\ \\
Joan Verdera\\
Departament de Matem\`{a}tiques\\
Universitat Aut\`{o}noma de Barcelona\\
08193 Bellaterra, Barcelona, Catalonia\\
{\it E-mail:} {\tt jvm@mat.uab.cat}
\end{tabular}

\end{document}